\documentclass{siamart0516}

\usepackage{lipsum}
\usepackage{amsfonts}
\usepackage{amsmath}
\usepackage{amssymb}
\usepackage{graphicx}
\usepackage{xcolor}
\usepackage{epstopdf}
\usepackage{algorithmic}
\usepackage{dsfont}
\usepackage{mathtools}
\usepackage{multirow}
\ifpdf
  \DeclareGraphicsExtensions{.eps,.pdf,.png,.jpg}
\else
  \DeclareGraphicsExtensions{.eps}
\fi

\newcommand{\TheTitle}{Existence of Weak Solutions for a pseudo-parabolic system Coupling Chemical Reactions, Diffusion and Momentum Equations}
\newcommand{\TheAuthors}{A. J. Vromans, A. A. F. van de Ven and A. Muntean}
\newcommand{\ShortTitle}{Exist. of weak sol. for coupled parab. and pseudo-parab. sys.}
\headers{\ShortTitle}{\TheAuthors}

\title{{\TheTitle}\thanks{Submitted to the editors February 7th, 2017.
\funding{This work was funded by the Netherlands Organization for Scientific Research (NWO) under contract no.~NWO-MPE 657.000.004.}}}

\author{
  Arthur J. Vromans$^{\dag\ddag}$
  \and
  Fons van de Ven\thanks{Centre for Analysis, Scientific Computing and Applications (CASA), Technische Universiteit Eindhoven, Eindhoven, The Netherlands
    (\email{a.j.vromans@tue.nl}, \email{A.A.F.v.d.Ven@tue.nl}).}
  \and
  Adrian Muntean\thanks{Department of Mathematics and Computer Science, Karlstads Universitet, Karlstad, Sweden (\email{adrian.muntean@kau.se}).}
}

\usepackage{amsopn}

\ifpdf
\hypersetup{
  pdftitle={\TheTitle},
  pdfauthor={\TheAuthors}
}
\fi

\DeclareMathOperator*{\esssup}{ess\,sup}
\newsiamthm{assumption}{Assumption}
\newsiamthm{claim}{Claim}
\crefname{assumption}{Assumption}{Assumptions}
\crefname{claim}{Claim}{Claims}
\begin{document}

\maketitle

\begin{abstract}
 We study the weak solvability of a nonlinearly coupled system of parabolic and pseudo-parabolic equations describing the interplay between mechanics, chemical reactions, diffusion and flow in a mixture theory framework. Our approach relies on suitable discrete-in-time energy-like estimates and discrete Gronwall inequalities. In selected parameter regimes, these estimates ensure the convergence of the Rothe method for the discretized partial differential equations.
\end{abstract}


\section{Introduction}$\;$\\
\label{s: sec1}
 We investigate the existence of weak solutions to a system of partial differential equations coupling chemical reaction, momentum transfer and diffusion, cast in the framework of mixture theory \cite{Bowen1980}. We use the Rothe method \cite{Kacur1985,Rothe1984} as main tool. For simplicity, we restrict ourselves to a model with a single non-reversible chemical reaction in a one-dimensional bounded spatial domain $[0,1]$ enclosed by unlimited (or instantly replenished) reservoirs of the reacting chemicals. The chemical reaction is of the $N$-to-$1$-type with the reacting chemicals consisting out of solids and a single fluid, while the produced chemical is a solid. New mathematical challenges arise due to the strong nonlinear coupling between all unknowns and their transport fluxes. \\

Evolution systems, in which chemical reactions, momentum transfer, diffusion and stresses interplay, occur practically in every physical or biological system where there is enough knowledge to describe completely the balances of masses and forces; see e.g. \cite{Chabaud2012,Eden2016,Fatima2014,Piatniski2015}. In all these situations, the interest lies in capturing  the effects flows have on deformations, deformations and chemical reactions on structures, and structures on chemical reactions and flow. In biology, such a system is used, for instance, to better understand and eventually forecast the plant growth and development \cite{Piatniski2015}. In structural engineering, one wants to delimit the durability of a concrete sample exposed to ambiental corrosion, for example sulfate attack in sewer pipes \cite{Fatima2014}. Our initial interest in this topic originates from mathematical descriptions of sulfate corrosion \cite{Bohm1998}. We have realized that the mathematical techniques used for a system describing sulfate attack [when within a porous media (concrete) sulfuric acid reacts with slaked lime to produce gypsum], could be equally well applied to some more general systems sharing similar features (e.g. types of flux couplings and nonlinearities).\\

At a general level, the system outlined in this paper is a combination of parabolic equations  of diffusion-drift type with production terms by chemical reactions and pseudo-parabolic stress equations based on viscoelastic terms. On their own, both  parabolic equations (cf. e.g. \cite{Evans2010,Lady1967,LaxMilgram1954}, and pseudo-parabolic equations (see e.g. \cite{BohmShowalter1985,FanPop2013,Ford1976,PtashnykPHD,Ptashnyk2007,Showalter1970}) are well-understood from mathematical and numerical analysis perspectives. However, coupling these objects leads to systems of equations with a less understood structure. Many systems in the literature seem similar to ours at a first glance. A coupling remotely resembling our case appears in \cite{AbelsLiu2016}, but with different nonlinear terms, others like those in \cite{AbelsLiu2016,Chabaud2012} do not have the pseudo-parabolic part, \cite{Eden2016,Fatima2014} refer to a different domain situation, while in  \cite{Piatniski2015} higher-order derivatives are involved.\\

Due to the strong coupling present in our system, we chose to investigate in this paper the simplest case: a one-dimensional bounded domain, benefiting this way of an easier control of the nonlinearities by relying on the embedding $H^1\hookrightarrow L^\infty$ within a decoupling strategy of the model equations inspired by the method of Rothe. The study of the multidimensional case will be done elsewhere.\\

We apply our techniques to a general system introduced in \Cref{s: sec2}, which covers e.g. mathematical models describing sulfate attack on concrete. In this section, we also introduce a set of assumptions based on which the existence of weak solutions can be proven. In our setting, the parabolic equations contain only coupling terms consisting of time-derivative terms of the unknowns of the pseudo-parabolic equations, while the pseudo-parabolic equations contain only coupling terms through Lipschitz-like non-linearities coupling back to the parabolic part of the system. In \Cref{s: sec3}, we apply a time discretization decoupling the evolution system, inspired by the method of Rothe,  such that the Lipschitz functions are evaluated at a different time-slice than the unknowns involved in the pseudo-parabolic part.
The decoupled pseudo-parabolic equations can now be solved given the solution of the parabolic system posed at the previous time slice, while the new parabolic part can be solved with the just obtained solution of the pseudo-parabolic equations. The discrete-in-time a priori energy-like estimates are derived  in \Cref{s: sec4} by testing the discretized system with suitable functions leading to quadratic terms and then by applying the discrete Gronwall lemma to the resulting quadratic inequalities. Based on these a-priori estimates, we show in \Cref{s: sec5}  that our assumptions stated in  \Cref{s: sec2} are  valid in certain parameter regions. Furthermore, based on our a-priori estimates, we prove in \Cref{s: sec6} that the linear interpolation functions of the solutions to the discrete system converge strongly to a weak solution of the original system.

\section{Description of the system}$\;$\\
\label{s: sec2}
We define our system on a time-space domain $[t_0,T]\times[0,1]$, where $t_0$ is the initial time and $T$ is the final time defined at a later stage by \Cref{a: ass3}. The unknowns of our system are two vector functions, $\phi: ([t_0,T]\times[0,1])^d\rightarrow\mathbf{R}^d$ and $w: ([t_0,T]\times[0,1])^{d-1}\rightarrow\mathbf{R}^{d-1}$, and one scalar function $v: [t_0,T]\times[0,1]\rightarrow\mathbf{R}$. The vector $\phi$ consists of the volume fractions of the $d$ chemical components active in a chemical reaction mechanism of redox type. The vector $w$ refers to the displacements of the mixture components with respect to a reference coordinate system. The scalar function $v$ denotes the velocity of the single chemical fluid. We identify the different components of the vectors with the different chemicals and use the following convention. The subscript 1 is related to the produced chemical. The subscript $d$ is related to the fluid. All the other subscripts are related to the remaining solid chemicals. The unknowns interplay in the following system of the evolution equations:  For $l\in\{1,\ldots,d\}$, $m\in\{1,\ldots,d-1\}$, we have
\begin{subequations}
\begin{align}
\partial_t\phi_l-\delta_l\partial_z^2\phi_l+I_l(\phi)\partial_z\left(\Gamma(\phi)v\right)+\sum_{m=1}^{d-1}\sum_{i,j=0}^1\partial_z^i\left(B_{lijm}(\phi)\partial_t^jw_m\right)&=G_{\phi,l}(\phi)\label{eq: sys1}\\
\partial_z\left(\Gamma(\phi)v\right)+\sum_{m=1}^{d-1}\sum_{j=0}^1\partial_z\left(H_{jm}(\phi)\partial_t^jw_m\right)&=G_{v}(\phi)\label{eq: sys2}\\
\partial_tw_m-D_m\partial_z^2w_m-\gamma_m\partial_z^2\partial_tw_m+F_m(\phi)v&\label{eq: sys3}\\
+\sum_{j=1}^{d-1}\sum_{\substack{i+n=0\\i,n\geq0}}^1\partial_z\left(E_{minj}(\phi)\partial_z^i\partial_t^nw_j\right)&=G_{w,m}(\phi)\nonumber
\end{align}
\end{subequations}
with functions $I_l,\,\Gamma,\,B_{lijm},\,H_{jm},\,F_m,\,E_{minj},\,G_{\phi,l},\,G_v,\,G_{w,m}\in W^{1,\infty}\!\!\left((0,1)^d\right)$ and\\
constants $\delta_l,\,D_m,\,\gamma_m\in\mathbf{R}_+$. Furthermore, abuse notation with $\|f(\cdot)\|_{W^{1,\infty}\left((0,1)^d\right)}\leq f\in\mathbf{R}_+$. Notice that this system must satisfy the conditions $\sum_{l=1}^d\phi_l=1$, the fundamental equation of fractions, which allows for the removal of the $l=d-1$ component of \Cref{eq: sys1}.\\
We assume the volume fractions are insulated at the boundary, thus implying $\partial_z\phi=0$ at the boundaries $z=0$ and $z=1$. The boundary at $z=0$ is assumed to be fixed, while the boundary at $z=1$ has a displacement $W(t) = h(t)-1$, where $h(t)$ is the height of the reaction layer at the present time $t$. The Rankine-Hugoniot relation, see e.g. \cite{Muntean2015}, states that the velocity of a chemical at a boundary is offset from the velocity, $U$, of the boundary by influx or outflux of the chemical, i.e.
\begin{equation}
\begin{dcases}
\phi_m\left(U-\partial_tw_m\right)\cdot\hat{n} = \hat{J}_m\mathcal{L}\left(\phi_{m,res}-\phi_m\right)\cr
\,\,\quad\phi_d\left(U-v\right)\cdot\hat{n} = \hat{J}_d\mathcal{L}\left(\phi_{d,res}-\phi_d\right)
\end{dcases}\nonumber
\end{equation}
In general the function $\mathcal{L}(\cdot)$ denotes the concentration jump across the boundary. However, we assume the boundary to be semi-permeable in such a way that only influx can occur. Hence $\mathcal{L}(f):=f\mathcal{H}(f)= f_+$ denotes the positive part of $f$. Furthermore, we assume that the fluid reservoir is at the boundary $z=1$: $\phi_{d,res}$ is positive at $z>1$, but 0 at $z<0$. The produced chemical does not have any reservoir at the boundaries. Therefore $\phi_{1,res}=0$ at both $z<0$ and $z>1$. The other chemicals have a reservoir below the $z=0$ boundary: $\phi_{m,res}$ is positive at $z<0$ and 0 at $z>1$ for $1<m<d$. We generalize the Rankine-Hugoniot relations by replacing $\phi_m$ with $H_{1m}(\phi)$ and $\phi_d$ with $\Gamma(\phi)$.\\
The influx due to the Rankine-Hugoniot relations shows that the displacement $\left.w_m\right|_{z=1}$ will not be equal to the boundary displacement $W(t)$. This will result in stresses, which we incorporate with a Robin boundary condition at these locations \cite[section 5.3]{ViFlFl99}. Collectively for all $t\in[t_0,T]$ these boundary conditions are
\begin{subequations}
\begin{align}
&\left\{\begin{array}{rll}
\left.\partial_z\phi_l\right|_{z=0}&=0,&l\neq d-1\cr
\left.\partial_z\phi_l\right|_{z=1}&=0,&l\neq d-1
\end{array}\right.\label{eq: phiBC}\\
&\left\{\begin{array}{rll}
\left.w_1\right|_{z=0}&=0\cr
\left.\partial_zw_1\right|_{z=1} &= A_1\left(\left.w_1\right|_{z=1}-W(t)\right)\cr
\left.H_{1m}(\phi)\partial_tw_m\right|_{z=0} &= \hat{J}_m\mathcal{L}\left(\phi_{m,res}-\left.\phi_m\right|_{z=0}\right),&1<m<d\cr
\left.\partial_zw_m\right|_{z=1} &= A_m\left(\left.w_m\right|_{z=1}-W(t)\right),&1<m<d\cr
\left.v\right|_{z=0}&=0\cr
\quad\left.\Gamma(\phi)\left(\partial_th(t)-v\right)\right|_{z=1}&= \hat{J}_d\mathcal{L}\left(\phi_{d,res}-\left.\phi_d\right|_{z=1}\right)
\end{array}\right.\label{eq: rankhugBC}
\end{align}
\end{subequations}
It is worth noting that  in the limit $|A_m|\rightarrow\infty$ one formally obtains the Dirichlet boundary conditions, which are the natural boundary conditions for this system from a physical perspective.\\
The initial conditions describe a uniform and stationary equilibrium solution at $t=t_0$:
\begin{equation}
\phi_l(t_0,z) = \phi_{l0}\qquad\text{and}\qquad w_m(t_0,z)=0\quad\text{for all }z\in[0,1].\label{eq: init}
\end{equation}
The collection of \Cref{eq: sys1,eq: sys2,eq: sys3,eq: phiBC,eq: rankhugBC,eq: init} forms our continuous system. Notice that this continuous system needs the pseudo-parabolic terms: our existence method (of weak solutions) fails when we choose $\gamma_m=0$. Moreover, the height function $h(t)$ in \Cref{eq: rankhugBC} cannot be chosen freely. Integration of \Cref{eq: sys2} in both space and time together with \Cref{eq: rankhugBC} will yield an ODE of $h(t)$, and, in special cases, a closed expression of $h(t)$. Furthermore, the initial conditions \cref{eq: init} do not contain a description of $v(t_0,z)$, because this function can be calculated explicitly: \Cref{eq: init} together with \Cref{eq: sys2,eq: sys3} define a subsystem of the variables $\frac{\partial w_m}{\partial t}(t,z),\, v(t,z)$ on $(t,z)\in\{t_0\}\times(0,1)$ with boundary conditions of \Cref{eq: rankhugBC}. This subsystem can be written in the form on to which \Cref{t: strong1} can be applied, which shows that there exists a unique solution of this subsystem in $(C^2(0,1))^d$ iff both $\Gamma(\phi(t_0,z))\neq 0$ and $H_{1m}(t_0,0)\neq0$ are satisfied.\\
$\;$\\
Our existence proof relies on the following set of assumptions:
\begin{assumption}\label{a: ass1}
$d\geq2$, $\hat{J}_d,\hat{J}_m\geq0$ and $\phi_{d,res},\phi_{m,res}\in[0,1]$ for $1\leq m<d$ and\\ \vphantom{a}\qquad\quad $\sum_{l=1}^d\phi_{l,res}=1$.\\
\vphantom{a}\qquad\quad All chemical reactions, where $d-1$ chemicals react to form 1 chemical, are\\ \vphantom{a}\qquad\quad allowed. The reacting solid chemicals ($1<m<d$) flow into the domain from\\ \vphantom{a}\qquad\quad a reservoir at $z<0$, while the reacting liquid chemical flows into the domain\\ \vphantom{a}\qquad\quad from a reservoir at $z>1$. Moreover, the reservoir volume fractions are in the\\ \vphantom{a}\qquad\quad physical range $[0,1]$.
\end{assumption}
\begin{assumption}\label{a: ass4}
For all $0\!<\!\alpha\!<\!1/d$ introduce $\mathcal{I}_\alpha$ as $(\alpha,1-(d-1)\alpha)$. Then\\
\vphantom{a}\qquad\quad$\Gamma_{\alpha}:=\inf\limits_{\phi\in\mathcal{I}_\alpha^d}\Gamma(\phi)\in\mathbf{R}_+$ and $H_{\alpha}:=\min\limits_{1\leq m<d}\inf\limits_{\phi\in\mathcal{I}_\alpha^d}H_{1m}(\phi)\in\mathbf{R}_+$.\\
\vphantom{a}\qquad\quad The velocity $v(t,z)$ is now guaranteed to be bounded if the other velocities $\partial_tw$\\
\vphantom{a}\qquad\quad are bounded. Moreover, $v(t_0,z)$ is now given by \Cref{t: strong1}.
\end{assumption}
Next to these assumptions we have additional, which will be introduced pointwise at the appropriate moment. The pointwise introduced assumptions only list necessary conditions at that moment. If a more stringent condition is needed, then a new assumption will be introduced. For completeness, we list here the reasons for introducing the additional assumptions in their most stringent form.\\
\Cref{a: ass5} is guaranteeing the pseudo-parabolicity. \Cref{a: ass7} guarantees that the initial volume fractions are physical and non-zero. \Cref{a: ass9} guarantees that a $L^2(t_0,T;H^1(0,1))$ bound for $v^k$ can be found. \Cref{a: ass6} gives the necessary upper bound for the time discretization time step $\Delta t$, for which \Cref{a: ass3} can be proven.\\
$\,$\\
Next to these assumptions, we want our solutions to be physical at almost every time $t$. So, the volume fractions must lie in $[0,1]$ and the velocity $v(t)$ must be both essentially bounded and of bounded variation, which implies that both $v$ and $\partial_zv$ must be functions in the Bochner space $L^2(t_0,T;L^2(0,1))$. However, the volume fractions $\phi$ cannot become $0$ without creating problems for the original Rankine-Hugoniot boundary conditions or allowing singularities in the domain implying $\phi\in((0,1))^d$. To this end, we introduce a new time interval for which all of these constraints hold and we claim that such an interval exists.
\begin{claim}\label{a: ass3}
 There exists a time domain $[t_0,T]$, a velocity $V>0$ and a volume\\
  \vphantom{a}\qquad\qquad\qquad fraction $\phi_{\min}\in(0,1/d)$ such that $$\qquad\qquad\qquad T = \inf\left\{t\left|\begin{array}{rl}\int_{t_0}^t\int_0^1|v(s,z)|^2\mathrm{d}z\mathrm{d}s&>V^2,\cr
    \int_{t_0}^t\int_0^1\left|\partial_z v(s,z)\right|^2\mathrm{d}z\mathrm{d}s&>V^2,\cr
    \inf\limits_{z\in(0,1)}\left\{\min\limits_{1\leq l\leq d}\phi_l(t,z)\right\}&<\phi_{\min}\end{array}\right.\right\},$$ \qquad \qquad\qquad which, in line with \Cref{a: ass1}, implies\\
    $$\qquad\qquad\qquad\phi_l(t,z)\in\overline{\mathcal{I}_{\phi_{\min}}}=\left[\phi_{\min},1-(d-1)\phi_{\min}\right]\subset(0,1)$$
    \vphantom{a}\qquad\qquad\qquad for all $(t,z)\in[t_0,T]\times[0,1]$ and $1\leq l\leq d$.
\end{claim}
This claim guarantees that all chemicals are omnipresent, while velocities and deformations remain bounded. This claim can be related to the theories on parabolic and pseudo-parabolic equations. The claim, which will be proven in \Cref{s: sec5}, mimics the necessary $L^2(0,1)$ and $H^1(0,1)$ regularity of the coefficients in the parabolic and pseudo-parabolic equation theory. Combining the claim with \Cref{a: ass4} directly introduces the constants $\Gamma_{\phi_{\min}}$ and $H_{\phi_{\min}}$ respectively as lower bounds of $\Gamma(\phi)$ and all $H_{1m}(\phi)$ for $\phi\in\overline{\mathcal{I}_{\phi_{\min}}}^d$.\\
$\;$\\
In this paper we shall prove the existence of a volume fraction $\phi_{\min}\in(0,1/d)$, a velocity $V>0$ and a nonempty time interval $[t_0,T]$ such that a weak solution of the continuous system exists if \Cref{a: ass1,a: ass4,a: ass5,a: ass6,a: ass7,a: ass9} are satisfied. We will prove this statement in 4 steps:
 \begin{itemize}
 \item First, we discretise the continuous system in time with a regular grid of step size $\Delta t$, apply a specific Euler scheme and prove that this new discretised system can be solved iteratively in the classical sense at each time slice.
 \item Second, we make a weak version of the discretised system and prove that there exists a weak solution of the continuous system. This will be done by choosing specific test functions such that we obtain quadratic inequalities. By application of Young's inequality and using the Gronwall lemmas we obtain the energy-like estimates called the a priori estimates, which are step size $\Delta t$ independent upper bounds of the Sobolev norms of the weak solutions.
 \item Third, we prove \Cref{a: ass3} by showing that the upper bounds of the a priori estimates are increasing functions of $T-t_0$ and $V$ that have to satisfy specific upper bounds in order to guarantee the validity of \Cref{a: ass3}. Then, in certain parameter regions, regions in $(T-t_0,V)$-space exist for which \Cref{a: ass3} holds.
 \item Fourth, we will introduce temporal interpolation functions $\overline{u}(t) = u^{k}$ and $\hat{u}(t) = u^{k-1}+(t-t_{k-1})\mathcal{D}_{\Delta t}^k(u)$ to justify the existence of functions on $[t_0,T]\times[0,1]$ and we will show that the weak limit for $\Delta t\downarrow0$ exists and is a weak solution of the continuous system.
 \end{itemize}
\section{The discretised system and its classical solution}$\;$\\
\label{s: sec3}
We discretise time with regular temporal grid $t_k = t_0+k\Delta t$ for $\Delta t>0$. This discretization is applied to \Cref{a: ass3} with the infimum replaced with a minimum, the time $t$ replaced with the discrete time $t_k$ and the time integrals replaced with Riemann sums. Consequently, the time $T$ is now dependent on the discretization. To highlight this fact we introduce the notation $T_{\Delta t}$ for the time $T$ of \Cref{a: ass3} with regular temporal grid $t_k = t_0+k\Delta t$.\\
$\;$\\
We discretise the continuous system in such a way that the equations become linear elliptic equations with respect to evaluation at time $t_k$, and only contain two time evaluations: at $t_k$ and $t_{k-1}$. The time derivative $\partial_tu$ is replaced with the standard first order finite difference $\mathcal{D}_{\Delta t}^k(u):=\frac{u^k-u^{k-1}}{\Delta t}$, where $u^k(z) := u(t_k,z)$. The discretised system takes the form
\begin{subequations}
\begin{align}
\mathcal{D}_{\Delta t}^k(\phi_l)-\delta_l\partial_z^2\phi_l^k+I_l(\phi^{k-1})\partial_z\left(\Gamma(\phi^{k-1})v^{k-1}\right)\label{eq: disc1}\\
+\sum_{m=1}^{d-1}\sum_{i=0}^1\partial_z^i\left(B_{li0m}(\phi^{k-1})w_m^{k-1}+B_{li1m}(\phi^{k-1})\mathcal{D}_{\Delta t}^k(w_m)\right)&=G_{\phi,l}(\phi^{k-1})\cr
\sum_{m=1}^{d-1}\partial_z\left(H_{0m}(\phi^{k-1})w_m^{k-1}+H_{1m}(\phi^{k-1})\mathcal{D}_{\Delta t}^k(w_m)\right)\label{eq: disc2}\\
+\partial_z\left(\Gamma(\phi^{k-1})v^k\right)&=G_{v}(\phi^{k-1})\cr
\mathcal{D}_{\Delta t}^k(w_m)-D_m\partial_z^2w_m^k-\gamma_m\partial_z^2\mathcal{D}_{\Delta t}^k(w_m)+F_m(\phi^{k-1})v^{k-1}\label{eq: disc3}\\
+\sum_{j=1}^{d-1}\sum_{i=0}^1\partial_z\left(E_{mi0j}(\phi^{k-1})\partial_z^iw_j^{k-1}+E_{m01j}(\phi^{k-1})\mathcal{D}_{\Delta t}^k(w_j)\right)&=G_{w,m}(\phi^{k-1})\nonumber
\end{align}
\end{subequations}
with boundary conditions
\begin{subequations}
\begin{align}
&\left\{\begin{array}{rll}
\left.\partial_z\phi_l^k\right|_{z=0}&=0,&l\neq d-1\quad\cr
\left.\partial_z\phi_l^k\right|_{z=1}&=0,&l\neq d-1\quad
\end{array}\right.\label{eq: phiBCD}\\
&\left\{\begin{array}{rl}
\left.H_{1m}(\phi^{k-1}|_{z=0})\mathcal{D}_{\Delta t}^k(w_m)\right|_{z=0}\!\!\!&= \hat{J}_m\mathcal{L}\left(\phi_{m,res}-\left.\phi_m^{k-1}\right|_{z=0}\right)\cr
\left.\partial_zw_m^k\right|_{z=1}\!\!\!&= A_m\left(\left.w_m^k\right|_{z=1}-\mathcal{W}^k\right)\cr
\left.v^k\right|_{z=0}\!\!\!&=0\cr
\quad\left.\Gamma\left(\phi^{k-1}|_{z=1}\right)\left(\mathcal{D}_{\Delta t}^k(\mathcal{W})-v^{k-1}\right)\right|_{z=1}\!\!\!&= \hat{J}_d\mathcal{L}\left(\phi_{d,res}-\left.\phi_d^{k-1}\right|_{z=1}\right)\quad,
\end{array}\right.\label{eq: rankhugBCD}
\end{align}
\end{subequations}
where $\mathcal{W}^k:=W(t_k)$, and $\hat{J}_1=0$ in accordance with the first boundary condition of \cref{eq: rankhugBC}, and initial conditions \cref{eq: init}.\\
A powerful property of this discretised system is its sequential solvability at time $t_k$: the existence of a natural hierarchy in attacking this problem. First we obtain results for \Cref{eq: disc3}, then we use these results to obtain similar results for both \Cref{eq: disc1,eq: disc2}. Moreover, the structure of the discretised system is that of an elliptic system. Hence, the general existence theory for elliptic systems can be extended directly to cover our situation:
\begin{theorem}
\label{t: strong1}
Let $\Omega\subset\mathbf{R}$ bounded, $A_\pm,B_\pm,C_\pm\in M(\mathbf{R}^n,\mathbf{R}^n)$, $D$ a diagonal positive definite matrix of size $n\times n$, and $E_{ij},f_i\in (L^2(\Omega))^n$ for all $i,j\in\{1,\ldots,n\}$. There exists a unique solution $u\in (C^2(\Omega))^n$ of the system
\begin{equation}
\begin{dcases}\partial_z^2u-\partial_z(E(z)u)-Du=f &\text{ on }(x_-,x_+)=:\Omega\cr
A_\pm u(x_\pm)+B_\pm u'(x_\pm)= C_\pm.&
\end{dcases}\nonumber
\end{equation}
Moreover, if the conditions
\begin{equation}
\left|\begin{array}{cc}
A_++B_+E(x_+)&B_+\cr
A_-+B_-E(x_-)&B_-
\end{array}\right|\neq0\qquad\text{ and }\qquad\min_{z\in\Omega}|\mathrm{Tr}\left(E(z)\right)|>0\nonumber
\end{equation}
are satisfied, then the solution $u$ is given by
\begin{multline}
\begin{pmatrix}u(z)\cr U(z)\end{pmatrix} = \Psi(z)\int_{x_-}^z\Psi(s)\begin{pmatrix}0\cr f(s)\end{pmatrix}\mathrm{d}s+\begin{pmatrix}
A_++B_+E(x_+)&B_+\cr
A_-+B_-E(x_-)&B_-
\end{pmatrix}^{-1}\begin{pmatrix}C_+\cr C_-\end{pmatrix}\cr
+\begin{pmatrix}
A_++B_+E(x_+)&B_+\cr
0&0
\end{pmatrix}\Psi(x_+)\int_{x_-}^{x_+}\Psi(s)\begin{pmatrix}0\cr f(s)\end{pmatrix}\mathrm{d}s\nonumber
\end{multline}
with
\begin{equation}
U(z) = \partial_zu-E(z)u,\nonumber
\end{equation}
and with $\Psi(z)\in M(\mathbf{R}^n,\mathbf{R}^n)$ for all $z\in\Omega$ satisfying
\begin{equation}
\begin{dcases}
\partial_z\Psi(z) = \begin{pmatrix}
E(z)&\mathbb{I}_n\cr
D^{-1}&0
\end{pmatrix}\Psi(z) &\text{ for } z\in\Omega\cr
\,\,\Psi(x_-)= \mathbb{I}_n&
\end{dcases}\nonumber
\end{equation}
\end{theorem}
\begin{proof}
See \cite[p.130]{PolZait1995} for the general calculus result or see Chapter 6 of \cite{Evans2010} for the elliptic theory result. In specific: rewrite system in terms of $u$ and $U$. This system has a fundamental matrix $\Psi(z)$ yielding the above solution after satisfying boundary conditions.
\end{proof}
\begin{corollary}\label{c: strongdisc}
Let $\Delta t>0$. For all $t_0<t_k$ such that $\phi^{k-1}\in(0,1)^d$ there exists a unique solution $u^k:=(\phi^k,w^k,v^k)\in\left(C^2(0,1)\right)^{2d-1}\times C^1(0,1)$ of the system \cref{eq: disc1,eq: disc2,eq: disc3} with boundary conditions \cref{eq: phiBCD,eq: rankhugBCD} and initial conditions \cref{eq: init}.
\end{corollary}
\begin{proof}
The result follows with induction with respect to $k\geq0$ from applying \Cref{t: strong1} and using both $I_l,\,\Gamma,\,B_{lijm},\,H_{jm},\,F_m,\,E_{minj},\,G_{\phi,l},\,G_v,\,G_{w,m}\!\in\! W^{1,\infty}\!\!\left((0,1)^d\right)$ and $\delta_l,\,D_m,\,\gamma_m\in\mathbf{R}_+$.
\end{proof}
This result shows that there exists a solution of the discrete system even if the solution does not satisfy \Cref{a: ass1,a: ass3,a: ass4,a: ass5,a: ass6,a: ass7,a: ass9}. The solution might therefore be non-physical. Furthermore these solutions might not have a weakly convergent limit as $\Delta t\downarrow0$. We will use a weak solution framework to show the existence of physical solutions for which the weak convergence as $\Delta t\downarrow0$ does exist.
\section{A priori estimates}$\;$\\
\label{s: sec4}
The estimates in this section rely on the validity of \Cref{a: ass3}. This validity will be proven in \Cref{s: sec5}. Notice that from this moment onwards the notation $l\neq d-1$ denotes $l\in\{1,\ldots,d-2,d\}$.\\
We create a weak form of the discretised system by multiplying the equations with a function in $H^1(0,1)$, integrating over $(0,1)$ and applying the boundary conditions where needed. To this end we test \Cref{eq: disc1} with $\phi_l^k$ and $\mathcal{D}_{\Delta t}^k(\phi_l)$, and \Cref{eq: disc3} with $w_m^k$ and $\mathcal{D}_{\Delta t}^k(w_m)$. In \Cref{s: app0} these tests are evaluated in detail and it is shown there how we obtain the following quadratic inequalities:
\begin{multline}
\frac{1}{2}\mathcal{D}_{\Delta t}^k\left(\sum_{m=1}^{d-1}\|w_m\|_{L^2}^2+(\gamma_m+D_m)\|\partial_zw_m\|_{L^2}^2\right)\label{eq: wbounds}\\
+\sum_{m=1}^{d-1}\left[\left(1+\frac{\Delta t}{2}\right)\left\|\mathcal{D}_{\Delta t}^k(w_m)\right\|_{L^2}^2+\left[\gamma_m\left(1+\frac{\Delta t}{2}\right)+D_m\frac{\Delta t}{2}\right]\left\|\mathcal{D}_{\Delta t}^k\left(\partial_zw_m\right)\right\|_{L^2}^2\right]\cr
\leq K_{w0}+\!\!\sum_{m=1}^{d-1}\!\left[\vphantom{\frac{A}{A}} K_{w1m}\|w_m^k\|_{L^2}^2\!+\!K_{w2m}\|\partial_zw_m^k\|_{L^2}^2\!+\!K_{w3m}\|w_m^{k-1}\|_{L^2}^2\!+\!K_{w4m}\|\partial_zw_m^{k-1}\|_{L^2}^2\right.\cr
\left.\vphantom{\frac{A}{A}}+\!K_{w5m}\left\|\mathcal{D}_{\Delta t}^k(w_m)\right\|_{L^2}^2\!+\!K_{w6m}\left\|\mathcal{D}_{\Delta t}^k(\partial_zw_m)\right\|_{L^2}^2\right]+K_{w7}\|v^{k-1}\|_{L^2}^2+K_{w8}\|\partial_zv^{k-1}\|_{L^2}^2,\!\!\!
\end{multline}
\begin{multline}
\mathcal{D}_{\Delta t}^k\left(\|\phi_l\|_{L^2}^2\right)+2\delta_l\|\partial_z\phi_l^k\|_{L^2}^2+\Delta t\left\|\mathcal{D}_{\Delta t}^k(\phi_l)\right\|_{L^2}^2\label{eq: phibounds}\\
\leq K_{a\phi0}+K_{a\phi1}\|\partial_zv^{k-1}\|_{L^2}^2+K_{a\phi2l}\|\phi_l^{k}\|_{L^2}^2+\sum_{n\neq d-1}\left[K_{a\phi3ln}\|\partial_z\phi_n^{k-1}\|_{L^2}^2\right]\cr
+\sum_{m=1}^{d-1}\sum_{i=0}^1\left[K_{a\phi(4+i)m}\left\|\partial_z^iw_m^{k-1}\right\|_{L^2}^2+K_{a\phi(6+i)m}\left\|\mathcal{D}_{\Delta t}^k(\partial_z^iw_m)\right\|_{L^2}^2\right],
\end{multline}
and
\begin{multline}
\mathcal{D}_{\Delta t}^k\left(\sum_{l\neq d-1}\|\partial_z\phi_l\|_{L^2}^2\right)+\sum_{l\neq d-1}\left[\frac{2}{\delta_l}\left\|\mathcal{D}_{\Delta t}^k(\phi_l)\right\|_{L^2}^2+\Delta t\left\|\mathcal{D}_{\Delta t}^k\left(\partial_z\phi_l\right)\right\|_{L^2}^2\right]\label{eq: phibounds1}\\
\leq K_{b\phi0}+K_{b\phi1}\|\partial_zv^{k-1}\|_{L^2}^2+\sum_{l\neq d-1}\!\left[ K_{b\phi2l}\left\|\mathcal{D}_{\Delta t}^k(\phi_l)\right\|_{L^2}^2+K_{b\phi3l}^{k-1}\|\partial_z\phi_l^{k-1}\|_{L^2}^2\right]\cr
+\sum_{m=1}^{d-1}\sum_{i=0}^1\!\left[K_{b\phi(4+i)m}\left\|\partial_z^iw_m^{k-1}\right\|_{L^2}^2+K_{b\phi(6+i)m}\left\|\mathcal{D}_{\Delta t}^k(\partial_z^iw_m)\right\|_{L^2}^2\right].
\end{multline}
The constants $K_{index}$ for $x\in\{a,b\}$ can be found in \Cref{s: app0} as \Cref{eq: kw0,eq: kw1,eq: kw2,eq: kw3,eq: kw4,eq: kw5,eq: kw6,eq: kw7,eq: kw8,eq: kw8} and \Cref{eq: xkphi0,eq: xkphi1,eq: xkphi2,eq: xkphi3a,eq: xkphi3b,eq: xkphi4,eq: xkphi5,eq: xkphi6,eq: xkphi7}.\\
$\;$\\
We are now able to apply two versions of the Discrete Gronwall lemma. The 1st version (of \Cref{l: 1gron}) will be applied to \Cref{eq: wbounds,eq: phibounds}, while the 2nd version (of \Cref{l: 1gron2}) will be applied to \Cref{eq: phibounds1}.
\begin{lemma}[1st Discrete Gronwall lemma]
\label{l: 1gron}
Suppose $h\in(0,H)$. Let $(x^k)$, $(y^{k+1})$ and $(z^k)$ for $k = 0,1,\ldots$ be sequences in $\mathbf{R}_+$ satisfying
\begin{equation}
\label{e: 1gron}
y^k+\frac{x^k-x^{k-1}}{h}\leq A+z^{k-1}+Bx^k+Cx^{k-1}\quad\text{and}\quad\sum_{j=0}^{k-1}z^jh\leq Z
\end{equation}
for all $k=1,\ldots$ with constants $A$,$B$,$C$ and $Z$ independent of $h$ satisfying
\begin{equation}
A>0,\:\:\:Z>0,\:\:\:B+C>0,\:\:\:\text{and}\:\:\:BH\leq0.6838,\nonumber
\end{equation}
then
\begin{subequations}
\begin{align}
x^k&\leq \left(x^0+Z+A\frac{C+1.6838B}{C+B}kh\right)e^{\left(C+1.6838B\right)kh}\quad\text{and}\label{e: 1gron1}\\
\sum_{j=1}^ky^jh&\leq \left(x^0+Z+Ahk\right)e^{\left(C+1.6838B\right)kh}.\label{e: 1gron2}
\end{align}
\end{subequations}
\end{lemma}
\begin{proof}
We rewrite \Cref{e: 1gron} such that $x^k$ is on the left-hand side and $x^{k-1}$ is on the right-hand side. We can discard the $y^k$ term since it is always positive. The partial sum of the geometric series yields
\begin{equation}
x^k\leq\left(x^0+\frac{Z}{1+Ch}+\frac{A}{B+C}\right)\left(\frac{1+Ch}{1-Bh}\right)^k-\frac{A}{B+C}\nonumber
\end{equation}
from which we obtain \Cref{e: 1gron1} by applying the inequalities
\begin{equation}
1+a\leq e^a\leq 1+ae^a\text{ for }a\geq0\quad\text{ and }\quad \frac{1}{1-a}\leq e^{a+a^2}\text{ for }0\leq a\leq 0.6838.
\end{equation}
With $Bh-1\leq BH-1<0$ we rewrite \Cref{e: 1gron} into
\begin{equation}
\sum_{j=1}^ky^jh\leq Ahk+Z+(1+Ch)x^0+(C+B)h\sum_{j=1}^{k-1}u^j.\nonumber
\end{equation}
We insert \Cref{e: 1gron1} for $x^j$ and use $j<k$ for the factor in brackets. Then the sum over exponentials can be seen as a partial sum of a geometric series, yielding
\begin{equation}
\sum_{j=1}^ky^jh\leq x^0+Z+Ahk+(C+B)h\left(x^0+Z+A\frac{C+1.6838B}{C+B}kh\right)\frac{e^{(C+1.6838B)kh}-1}{e^{(C+1.6838B)h}-1}.\nonumber
\end{equation}
With $1/(e^a-1)\leq 1/a$ for $a\geq0$ one immediately obtains \Cref{e: 1gron2}.
\end{proof}
\begin{lemma}[2nd Discrete Gronwall lemma]
\label{l: 1gron2}
Let $c>0$ and $(y_k)$, $(g_k)$ be positive sequences satisfying
\begin{equation}
y_k\leq c+\sum_{0\leq j<k} g_jy_j\quad\text{ for }k\geq0,
\end{equation}
then
\begin{equation}
y_k\leq c\exp\left(\sum_{0\leq j<k} g_j\right)\quad\text{ for }k\geq0.
\end{equation}
\end{lemma}
\begin{proof}
The proof can be found in \cite{Holte2009}.
\end{proof}
We introduce a set of new constants, which aid us in the application of the Gronwall lemmas:
\begin{definition}\label{d: def1}
\begin{subequations}
\begin{align}
A&:=\max_{t_0\leq t_k\leq T_{\Delta t}}K_{w0}^k,\cr
\frac{B}{2}&:=\max_{m}\left\{K_{w1m},\frac{K_{w2m}}{\gamma_m+D_m}\right\},\cr
\frac{C}{2}&:=\max_{m}\left\{K_{w3m},\frac{K_{w4m}}{\gamma_m+D_m}\right\},\cr
\tilde{\mathcal{C}}^2&:=\frac{\left((K_{w7}+K_{w8})V^2+A\frac{C+1.6838B}{C+B}(T_{\Delta t}-t_0)\right)e^{(C+1.6838B)(T_{\Delta t}-t_0)}}{\min\limits_{m}\left\{1-K_{w5m},\gamma_m-K_{w6m}\right\}},\cr
g_{bk}&:=\max\limits_{l\neq d-1}K_{b\phi3l}^k\Delta t,\cr
K_{\phi3l4}& :=\sum_{m=1}^{d-1}\frac{K_{\phi3l2m}+K_{\phi3l3m}}{\min\{1-K_{w5m},\gamma_m-K_{w6m}\}},\cr
c_b&:=K_{b\phi0}(T_{\Delta t}-t_0)+K_{b\phi1}V^2+\max_{m}\left\{K_{b\phi4m},\frac{K_{b\phi5m}}{\gamma_m+D_m}\right\}\tilde{\mathcal{C}}^2(T_{\Delta t}-t_0)\cr
&\qquad+\max_{m}\left\{\frac{K_{b\phi6m}}{1-K_{w5m}},\frac{K_{b\phi7m}}{\gamma_m-K_{w6m}}\right\}\tilde{\mathcal{C}}^2,\cr
D_b&:=\max\limits_{l\neq d-1}\left\{2K_{\phi3l1}V^2+K_{\phi3l4}\tilde{\mathcal{C}}^2\right\}.\nonumber
\end{align}
\end{subequations}
The constants $K_{index}$ not introduced here can be found in \Cref{eq: kw0,eq: kw1,eq: kw2,eq: kw3,eq: kw4,eq: kw5,eq: kw6,eq: kw7,eq: kw8,eq: kw8} and \Cref{eq: xkphi0,eq: xkphi1,eq: xkphi2,eq: xkphi3a,eq: xkphi3b,eq: xkphi4,eq: xkphi5,eq: xkphi6,eq: xkphi7} of \Cref{s: app0}. Moreover, introduce the constants
\begin{subequations}
\begin{align}
A_a&:=K_{a\phi0}+c_b\exp(D_b)\max_{l,j\neq d-1}K_{a\phi3lj}+\max_{m}\left\{K_{a\phi4m},\frac{K_{a\phi5m}}{\gamma_m+D_m}\right\}\tilde{\mathcal{C}}^2,\cr
Z_a&:=K_{a\phi1}V^2+\max_{m}\left\{\frac{K_{a\phi6m}}{1-K_{w5m}},\frac{K_{a\phi7m}}{\gamma_m-K_{w6m}}\right\}\tilde{\mathcal{C}}^2,\cr
D_a&:=\max_{l\neq d-1}K_{a\phi2l}.\nonumber
\end{align}
\end{subequations}
\end{definition}
With these new notations, we obtain several simple expressions for the upper bounds by applying the Gronwall inequalities. However these expressions are only valid if the following assumptions are met.
\begin{assumption}\label{a: ass5} Let $B$ and $C$ be given by \Cref{d: def1} and $\mathcal{M}_{m01j}$ as\\
\vphantom{a}\qquad\quad introduced in \Cref{eq: Mbound}. Assume $B+C>0$ and
\begin{multline}
\qquad\mathcal{M}_{m11j}+\mathcal{M}_{j11m}+\mathcal{M}_{m01j}\sqrt{\gamma_m}=\cr
\gamma_m|A_m|+\gamma_j|A_j|+E_{m01j}(1+\sqrt{\gamma_m})+E_{j01m}<\frac{2\gamma_m}{d-1}.\nonumber
\end{multline}
\vphantom{a}\qquad\quad The second condition guarantees the pseudo-parabolicity, because it guarantees\\ \vphantom{a}\qquad\quad the conditions $1>K_{w5m}$ and $\gamma_m>K_{w6m}$ in \Cref{c: 1gron}.
\end{assumption}
\begin{assumption}\label{a: ass6a} Let $B$ be given by \Cref{d: def1}. Assume $H=0.6838/B$.
\end{assumption}
With these two assumptions the Gronwall inequalities imply the following upper bounds:
\begin{corollary}\label{c: 1gron}
Let $\Delta t\in(0,H)$. Introduce the sequences
\begin{subequations}
\begin{align}
x^k &= \sum_{m=1}^{d-1}\left[\|w_m^k\|_{L^2}^2+(\gamma_m+D_m)\|\partial_zw_m^k\|_{L^2}^2\right]\cr
y^k &= \sum_{m=1}^{d-1}\left[(1-K_{w5m})\left\|\mathcal{D}_{\Delta t}^k(w_m)\right\|_{L^2}^2+\left(\gamma_m-K_{w6m}\right)\left\|\mathcal{D}_{\Delta t}^k\left(\partial_zw_m\right)\right\|_{L^2}^2\right]\cr
z^k&= K_{w7}\|v^k\|_{L^2}^2+K_{w8}\|\partial_zv^k\|_{L^2}^2.\nonumber
\end{align}
\end{subequations}
The inequality
\begin{equation}
x^k\leq \left((K_{w7}+K_{w8})V^2+A\frac{C+1.6838B}{C+B}(t_k-t_0)\right)e^{(C+1.6838B)(t_k-t_0)}\nonumber
\end{equation}
is valid for all $t_k\in[t_0,T_{\Delta t}]$. Furthermore, $1>K_{w5m}$ and $\gamma_m>K_{w6m}$ and the inequality
\begin{equation}
\sum_{j=1}^ky^j\Delta t\leq \left((K_{w7}+K_{w8})V^2+A(t_k-t_0)\right)e^{(C+1.6838B)(t_k-t_0)}\nonumber
\end{equation}
is valid for all $t_k\in[t_0,T_{\Delta t}]$.
\end{corollary}
\begin{proof}
We first need to show that we can choose the $\eta$-indices such that $1>K_{w5m}$ and $\gamma_m>K_{w6m}$, as otherwise we have insufficient information to bound all terms. There are $d-1$ couplings between $K_{w5m}$ and $K_{w6m}$ through $\eta_{\mathcal{M}m01j1}$. These couplings only give $\mathcal{M}_{m01j}<2\sqrt{\gamma_m}/(d-1)$ and $\eta_{\mathcal{M}m01j1}=1/\sqrt{\gamma_m}$ as conditions. However $K_{w6m}$ is coupled to itself as well through $\eta_{\mathcal{M}m11j1}$. With $\eta_{\mathcal{M}m11j1}=1$ it is immediately clear that the constraint of \Cref{a: ass5} allows one to choose the remaining $\eta$-indices in $K_{w5m}$ and $K_{w6m}$ freely to satisfy the inequalities. Hence, $1>K_{w5m}$ and $\gamma_m>K_{w6m}$ is guaranteed by \Cref{a: ass5}.\\
Next we apply \Cref{l: 1gron} to \Cref{eq: wbounds} in the form of \Cref{e: 1gron} with $x^k$, $y^k$ and $z^k$ as above. With use of the identity $k\Delta t = t_k-t_0$ and with the initial condition  $w_m^0(z)=0$ for $z\in(0,1)$, yielding $x^0(z)=0$, the two inequalities of \Cref{c: 1gron} directly follow from \Cref{l: 1gron}.
\end{proof}
For the next result, we choose $\eta_{index}>0$ such that $2>K_{b\phi2l}\delta_l$ holds.
\begin{corollary}\label{c: 1gron2}
Let $t_0< t_k\leq T_{\Delta t}$. Introduce the sequence
\begin{subequations}
\begin{align}
x_b^k &= \sum_{l\neq d-1}\|\partial_z\phi_l^k\|_{L^2}^2\cr
y_b^k &= 2\sum_{l\neq d-1}\left[\left(\frac{2}{\delta_l}-K_{b\phi2l}\right)\left\|\mathcal{D}_{\Delta t}^k(\phi_l)\right\|_{L^2}^2\right].\nonumber
\end{align}
\end{subequations}
The inequalities
\begin{subequations}
\begin{align}
x_b^k&\leq c_b\exp(D_b)\cr
\sum_{j=1}^ky^j\Delta t&\leq c_b\left(1+D_b\exp(D_b)\right)\nonumber
\end{align}
\end{subequations}
are valid for all $t_k\in(t_0,T]$.
\end{corollary}
\begin{proof}
 We can always choose $\eta_{index}>0$ in $K_{\phi2l}$ such that $2>\delta_lK_{\phi2l}$ holds. With the newly introduced notations $c_b$, $D_b$ and $g_{bn}$, and with the initial condition $\partial_z\phi^0(z)=0$ for $z\in(0,1)$, yielding $x_b^0=0$, we can sum inequality \cref{eq: phibounds1} and rewrite it into the inequality
\begin{equation}
x_b^k\leq x_b^k+\sum_{j=0}^{k-1}y_b^j\Delta t\leq c_b+x_b^0+\sum_{j=0}^{k-1}g_{bj}x_b^j=c_b+\sum_{j=0}^{k-1}g_{bj}x_b^j.\nonumber
\end{equation}
With $\sum_{j=0}^{k-1}g_{bj}\leq D_b$, direct application of \Cref{l: 1gron2} yields the bound of $x_b^k$. Substituting this upper bound for $x_b^k$ will directly yield the upper bound of $\sum_{j=1}^ky_b^j\Delta t$.
\end{proof}
We need an extra upper bound on $H$ to guarantee the successful application of the Discrete Gronwall lemma for $\phi_l^k$.
\begin{assumption}\label{a: ass6b} Let $B$ and $D_a$ be given by \Cref{d: def1}. Assume
$$H=\min\left\{\frac{0.6838}{B},\frac{0.6838}{D_a}\right\}.$$
\end{assumption}
Now an upper bound for $\|\phi_l^k\|_{L^2}$ can be determined.
\begin{corollary}
\label{c: 1gron3}
For $\Delta t\in(0,H)$ the sequence $\|\phi_l^k\|_{L^2}^2$ satisfies the inequality
\begin{equation}
\|\phi_l^k\|_{L^2}^2\leq\left(\phi_{l0}^2+Z_a+1.6838A_a(t_k-t_0)\right)e^{1.6838D_a(t_k-t_0)}\nonumber
\end{equation}
for all $t_k\in[t_0,T_{\Delta t}]$ and all $l\neq d-1$, if \Cref{a: ass5,a: ass6b} hold.
\end{corollary}
\begin{proof}
We insert the bounds from \Cref{c: 1gron,c: 1gron2} in \Cref{eq: phibounds}, and we use the new notations $x_a^k:=\|\phi_l^k\|_{L^2}^2$, and
\begin{equation}
z_a^k:=K_{a\phi1}\|\partial_zv^{k-1}\|_{L^2}^2+\sum_{m=1}^{d-1}\sum_{i=0}^1K_{a\phi(6+i)m}\left\|\mathcal{D}_{\Delta t}^k(\partial_z^iw_m)\right\|_{L^2}^2,\nonumber
\end{equation}
to obtain
\begin{equation}
\mathcal{D}_{\Delta t}^k\left(x_a\right)\leq A_a+z_a^k+D_ax_a^k\nonumber
\end{equation}
from \Cref{e: 1gron}. Once more apply \Cref{l: 1gron} to obtain the result of \Cref{c: 1gron3}.
\end{proof}
\section{Proof of \Cref{a: ass3}}$\;$\\
\label{s: sec5}
The a priori estimates of the previous section depend heavily on \Cref{a: ass3}. This claim restricts the time interval $(t_0,T_{\Delta t})$ for which the physical volume fraction restriction  $\phi_j^k(z)\in[\phi_{\min},1-(d-1)\phi_{\min}]$ for all $j\neq d-1$, and the physical velocity restrictions $\sum_{j=0}^k\|v^{j}\|_{L^2}^2\Delta t\leq V^2$, and $\sum_{j=0}^k\left\|\partial_z v^{j}\right\|^2_{L^2}\Delta t\leq V^2$ are valid. We need to prove that the interval $(t_0,T_{\Delta t})$ is non-empty. On closer inspection, we see that \Cref{a: ass3} can be proven with upper bounds only.\\
\begin{lemma}
\label{l: constraint} Let $t_0\leq t_k=k\Delta t\leq T_{\Delta t}$. Let $\mathbf{P}_d$ be the set of cyclic permutations of $(1,\ldots,d)$. The constraints $\phi_l^k(z)\!\in\![\phi_{\min},\!1-(d-1)\phi_{\min}]$ for $l\neq d-1$, $\sum\limits_{j=0}^k\|v^{j}\|_{L^2}^2\Delta t\leq V^2$, and $\sum\limits_{j=0}^k\left\|\partial_z v^{j}\right\|^2_{L^2}\Delta t\leq V^2$ are implied by
\begin{subequations}
\begin{align}
  \sum_{j\neq d}\left\|\phi_{\alpha_j}^k\right\|_{H^1}^2&\leq\frac{1}{2}\left(\frac{1-\phi_{\min}}{C_\infty}\right)^2\text{ for all }\alpha\in\mathbf{P}_d,\cr
\text{and}\qquad\sum_{j=0}^k\left\|\partial_z v^{j}\right\|^2_{L^2}\Delta t&\leq V^2.\nonumber
\end{align}
\end{subequations}
\end{lemma}
\begin{proof}
The 3rd condition of \Cref{eq: rankhugBCD} allows the application of the Poincar\'{e} inequality, which in one-dimensional space immediately gives the inequality $\|v^{j}\|_{L^2}\leq\|\partial_zv^{j}\|_{L^2}$. With this relation we can reduce the number of constraints on $v^k$ from two to one.\\
For the constraints on $\phi_l^k$ we randomly pick an $\alpha\in\mathbf{P}_d$ and start with reversing Young's inequality on $\sum_{j\neq d}\left\|\phi_{\alpha_j}^k\right\|_{H^1}^2\leq\frac{1}{2}\left(\frac{1-\phi_{\min}}{C_\infty}\right)^2$, which yields $\sum_{j\neq d}\left\|\phi_{\alpha_j}^k\right\|_{H^1}\leq\frac{1-\phi_{\min}}{C_\infty}$. This inequality is transformed by the Sobolev embedding theorem into $\sum_{j\neq d}\left\|\phi_{\alpha_j}^k\right\|_{L^\infty}\leq1-\phi_{\min}$. With $\phi_{\alpha_j}^k\in C^2(0,1)$ for $j\neq d$, which is shown in \Cref{t: strong1}, we can change the $L^\infty$ norm into a proper supremum on $(0,1)$. Hence we obtain $\inf_{z\in(0,1)}\phi_{\alpha_d}^k\geq\phi_{\min}$ from the identity $1=\sum_{1\leq l\leq d}\phi_l^k$. Since $\alpha$ was chosen randomly, we conclude that this result holds for all $\alpha\in \mathbf{P}_d$. Hence, $\min\limits_{1\leq l\leq d}\inf\limits_{z\in(0,1)}\phi_{l}^k(z)\geq\phi_{\min}$. With the $d$ infima established it is immediately clear that the $d$ suprema follow automatically from the same identity.
\end{proof}
$\;$\\
The volume fraction identity $1=\sum_{l=1}^d\phi_l^k$ cannot be used directly to satisfy the desired bound of \Cref{l: constraint}. It would use a circular argument involving $\phi_{\min}$: proving the lower bound $\phi_{\min}$ by a more stringent upper bound found in \Cref{l: constraint} and proving that upper bound with the volume fraction identity and the lower bound $\phi_{\min}$.\\
Indirectly, the volume fraction identity can be used as is shown in \Cref{s: app0}, where a differential equation for $\phi_{d-1}^k$ is constructed and tested such that new inequalities \cref{eq: phid,eq: phidz} are found.\\
Thus, from these new inequalities \cref{eq: phid,eq: phidz}, \Cref{l: 1gron} together with \Cref{c: 1gron,c: 1gron2,c: 1gron3} yields an upper bound
\begin{equation}\label{eq: 1phid}
\left\|\phi_{d-1}^k\right\|_{H^1}^2\leq \mathcal{P}\left(T_{\Delta t}-t_0,V\right)
\end{equation}
 The function $\mathcal{P}(x,y)$ has $\mathcal{P}(0,y)=\phi_{(d-1)0}^2$ and it is a strictly increasing continuous function. Hence, for all cyclic permutations $\alpha$ of $(1,\ldots,d)$ there exist strictly increasing continuous functions $P_\alpha(T_{\Delta t}-t_0,V)$ with $P_\alpha(0,0)=\sum_{j=1}^{d-1}\phi_{\alpha_j0}^2$, such that
$$\sum_{j=1}^{d-1}\|\phi_{\alpha_j}^k\|_{H^1}^2\leq P_\alpha(T_{\Delta t}-t_0,V)$$ for all $t_0\leq t_k\leq T_{\Delta t}$.\\
$\;$\\
With this upper bound at hand, we are now able to show the existence of a non-empty time interval $(t_0,T)$ for which the volume fraction condition of \Cref{a: ass3} holds. Before we proceed, we need to introduce a geometric solid, the Steinmetz solid.
\begin{definition}\label{d: steinmetz}
Introduce a $d$-dimensional solid cylinder with central axis $l_c$ and radius $r$ as the set of points at a distance at most $r$ from the line $l_c$. Following \cite{Katunin16,KoLkTuUpHi13} we introduce the $d$-dimensional Steinmetz solid $\mathcal{S}_d(r)$ as the intersection of $d$ $d$-dimensional solid cylinders with radius $r$ such that the axes $l_c$ intersect orthogonally at the origin. In particular we assume the orientation of the Steinmetz solid to be such that the cylinder axes $l_c$ are spans of Cartesian basis elements in $\mathbf{R}^d$. Hence
\begin{equation}
\mathcal{S}_d(r)=\left\{(x_1,\ldots,x_d)\in\mathbf{R}^d\left|\quad\sum_{j\neq i}x_j^2\leq r^2\text{ for all }1\leq i\leq d\right.\right\}.\nonumber
\end{equation}
\end{definition}
With the Steinmetz solid we obtain conditions for $\phi_{\min}$, $V$ and $T-t_0$ to obtain $\phi_l^k(z)\in(\phi_{\min},1-(d-1)\phi_{\min})$ for all $l$ and all $t_0\leq t_k\leq T_{\Delta t}$. For $k=0$ these conditions impose constraints on the initial volume fractions. Therefore we need an extra assumption.
\begin{assumption}\label{a: ass7}Let $\phi^0=(\phi_{10},\ldots,\phi_{d0})\in\mathcal{S}$, where $$\mathcal{S} = \mathcal{S}_d\left(\frac{1-\phi_{\min}}{\sqrt{2}C_{\infty}}\right)\cap\left\{\phi^0\in\mathbf{R}^d\left|\sum_{l=1}^d\right.\phi_{l}^0=1\right\}\cap[\phi_{\min},\infty)^d,$$
\vphantom{a}\qquad\quad with $\mathcal{S}_d(\cdot)$ the Steinmetz solid as defined in \Cref{d: steinmetz} and assume
$$\phi_{\min}
\begin{dcases}
<1-\frac{\sqrt{2(d-1)}C_\infty}{d}&\text{ for }2<d\leq5,\cr
\leq\frac{1}{1+\sqrt{2(d-1)}C_\infty}&\text{ for }5<d.
  \end{dcases}$$
\end{assumption}
Denote the set of cyclic permutations of $(1,\ldots,d)$ with $\mathbf{P}_d$.
\begin{lemma}
\label{l: tint} There exists an open simply connected region $S\subset\mathbf{R}_+^2$ with $(0,0)\in \overline{S}$, such that
\begin{eqnarray}
(x,y)\in S&\Rightarrow&P_\alpha(x,y)<\frac{1}{2}\left(\frac{1-\phi_{\min}}{C_\infty}\right)^2\text{ for all }\alpha\in\mathbf{P}_d,\cr
(x,y)\in \partial S&\Rightarrow&P_\alpha(x,y)\leq\frac{1}{2}\left(\frac{1-\phi_{\min}}{C_\infty}\right)^2\text{ for all }\alpha\in\mathbf{P}_d,\cr
(x,y)\notin \overline{S}&\Rightarrow&P_\alpha(x,y)>\frac{1}{2}\left(\frac{1-\phi_{\min}}{C_\infty}\right)^2\text{ for at least one }\alpha\in\mathbf{P}_d,\nonumber
\end{eqnarray}
if \Cref{a: ass7} holds.
\end{lemma}
\begin{proof}
Remark that the largest hypercube contained in the Steinmetz solid $\mathcal{S}_d(r)$ has edges of length $2r/\sqrt{d-1}$, which is equal to the set of points with maximum norm at most $r/\sqrt{d-1}$. We have the value $r = \frac{1-\phi_{\min}}{\sqrt{2}C_{\infty}}$, for which the interior of the hypercube is given by the inequality $\max\limits_{1\leq l\leq d}\{\phi_{l0}\}<\frac{1-\phi_{\min}}{\sqrt{2(d-1)}C_{\infty}}$. The upper bound for all volume fractions is $1-\phi_{\min}$ and this is a bound on the hypercube, which implies $d\geq1+1/(2C_\infty^2)=5/4$. The Steinmetz solid exists only for $r>0$. However, the volume fractions have a lower bound $\phi_{\min}>0$, which gives the inequality $\frac{1-\phi_{\min}}{\sqrt{2(d-1)}C_\infty}\geq\phi_{\min}$ and therefore the condition $\phi_{\min}\leq\frac{1}{1+\sqrt{2(d-1)}C_\infty}$. Another bound follows from the lowest value for the maximum. The condition
\begin{equation}
\frac{1-\phi_{\min}}{\sqrt{2(d-1)}C_\infty}>\frac{1}{d}=\min_{\sum\limits_{l=1}^d\phi_{l0}=1}\max_{1\leq l\leq d}\{\phi_{l0}\}\nonumber
\end{equation}
implies the condition $\phi_{\min}<1-\frac{2\sqrt{d-1}C_\infty}{d}$. The lower bound must satisfy $\phi_{\min}>0$, which yields $d>2= C_\infty^2+C_\infty\Re\sqrt{C_\infty^2-2}$ and conveniently satisfies the $d\geq 5/4$ condition. The two upper bounds for $\phi_{\min}$ are not equal. The last bound is the lesser of the two for $d<2C_{\infty}^2+1=5$, while the second bound is the lesser of the two for $d>2C_{\infty}^2+1=5$. Hence, the minimum of the two upper bounds for $\phi_{\min}$ guarantees that there exists a subset of initial conditions in  $[\phi_{\min},1-(d-1)\phi_{\min}]^d$ such that for all $\alpha\in\mathbf{P}_d$ we have $P_\alpha(0,0)<\frac{1}{2}\left(\frac{1-\phi_{\min}}{C_\infty}\right)^2$. Because all $P_\alpha(x,y)$ are strictly increasing and continuous functions, we immediately obtain simply connected open sets $S_\alpha$ for each $P_\alpha(x,y)$ for which $(0,0)\in S_\alpha$ hold with the properties
\begin{eqnarray}
(x,y)\in S_\alpha&\Rightarrow&P_\alpha(x,y)<\frac{1}{2}\left(\frac{1-\phi_{\min}}{C_\infty}\right)^2,\cr
(x,y)\in \partial S_\alpha&\Rightarrow&P_\alpha(x,y)=\frac{1}{2}\left(\frac{1-\phi_{\min}}{C_\infty}\right)^2,\cr
(x,y)\notin \overline{S_\alpha}&\Rightarrow&P_\alpha(x,y)>\frac{1}{2}\left(\frac{1-\phi_{\min}}{C_\infty}\right)^2.\nonumber
\end{eqnarray}
 Hence, $S=\bigcap_{\alpha\in\mathbf{P}_d} S_\alpha$.
\end{proof}
Now, we only need to prove that we can choose a value $V>0$ for which \Cref{a: ass3} holds. To this end we use the function $Q_{\Delta t}(T_{\Delta t}-t_0,V^2)$ introduced in \Cref{s: app0} as an upper bound to $\sum_{j=0}^k\|\partial_zv^j\|_{L^2}^2\Delta t$. Hence, we need to prove $Q_{\Delta t}(T_{\Delta t}-t_0,V^2)\leq V^2$ for all $\Delta t\in(0,H)$.
\begin{assumption}\label{a: ass8} Assume
$$\frac{4d-2}{\Gamma_{\phi_{\min}}^2}\max_{m}\left\{\frac{H_{1m}^2}{\gamma_m-K_{w6m}}\right\}(K_{w7}+K_{w8})<1.\nonumber$$
\end{assumption}
\begin{lemma}\label{l: 1Vbound}
If \Cref{a: ass8} holds, then there exist $H^*,V^*>0$ and an open simply connected region $\mathcal{R}\subset\mathbf{R}_+^2$ with $(0,V^*)\in\overline{\mathcal{R}}$ such that
\begin{eqnarray}
(x,y)\in \mathcal{R}&\Rightarrow&Q_{H^*}(x,y)<y^2,\cr
(x,y)\in \partial \mathcal{R}&\Rightarrow&Q_{H^*}(x,y)=y^2,\cr
(x,y)\notin \overline{\mathcal{R}}&\Rightarrow&Q_{H^*}(x,y)>y^2.\nonumber
\end{eqnarray}
Furthermore $\overline{\mathcal{R}}\subset\{(x,y)\in\mathbf{R}_+^2|Q_{\Delta t}(x,y)\leq y^2\}$ for all $\Delta t\in(0,H^*)$.
\end{lemma}
\begin{proof}
Using \Cref{c: 1gron,c: 1gron2,c: 1gron3} we obtain:
\begin{multline}
Q_{\Delta t}(x,y)= \left\|\partial_zv^0\right\|_{L^2}^2\Delta t+\frac{4d-2}{\Gamma_{\phi_{\min}}^2}\left[\vphantom{\frac{H_{0l}^2}{\gamma_l+D_l}}8\Gamma^2C_\infty^2(d-1)c_b(x,y)\exp(D_b(x,y))y^2\right.\cr
+G_v^2x+x\max_m\left\{\frac{H_{0m}^2}{\gamma_m+D_m},4(d-1)C_\infty^2H_{0m}^2c_b(x,y)\exp(D_b(x,y))\right\}\hat{C}^2(x,y)\cr
\left.+\max_m\left\{\frac{H_{1m}^2}{\gamma_m-K_{w6m}},\frac{4(d-1)C_\infty^2H_{1m}^2c_b(x,y)\exp(D_b(x,y))}{1-K_{w5m}}\right\}\hat{C}^2(x,y)\right],\nonumber
\end{multline}
with
\begin{eqnarray}
\hat{C}^2(x,y) &:=&\left((K_{w7}+K_{w8})y^2+A\frac{C+1.6838B}{C+B}x\right)\exp((C+1.6838B)x)\cr
c_b(x,y) &:=&c_{b1}x+c_{b2}y^2+c_{b3}x\hat{C}^2(x,y)+c_{b4}\hat{C}^2(x,y)\cr
D_b(x,y) &:=&\max_{l\neq d-1}\{D_{b1l}y^2+D_{b2l}\hat{C}^2(x,y)\}.\nonumber
\end{eqnarray}
Introduce the constants
\begin{eqnarray}
\hat{c}_b&:=&c_{b2}+c_{b4}(K_{w7}+K_{w8}),\cr
\hat{D}_b&:=&\max_{l\neq d-1}\{D_{b2l}+D_{b3l}(K_{w7}+K_{w8})\}\nonumber
\end{eqnarray}
satisfying $c_b(0,y) = \hat{c}_by^2$ and $D_b(0,y)=\hat{D}_by^2$.\\
Since $Q_{\Delta t}(x,y)$ is differentiable with strictly increasing positive derivative in $x$ and $y$ and strictly positive derivative in $\Delta t$, it is sufficient to prove that there exists a region $y\in(y^{**},y^{***})$ for which $Q_{H^*}(0,y)<y^2$. Because of the positive derivatives, this result guarantees for all $\Delta t\in(0,H^*)$ that there exists an open simply connected region such that $Q_{\Delta t}(x,y)< y^2$ in the region and $Q_{\Delta t}(x,y)\leq y^2$ on the boundary of the region.\\
The function $Q_{\Delta t}(0,y)$ satisfies the inequality
$$Q_{\Delta t}(0,y)\leq \mathbf{Q}_{\Delta t}(y) = Q_0\Delta t+Q_1y^2+Q_2y^4\exp(Q_3y^2)$$
with
\begin{eqnarray}
Q_0 &=& \|\partial_z v^0\|_{L^2}^2\cr
Q_1 &=& \frac{4d-2}{\Gamma_{\phi_{\min}}^2}\max_{m}\left\{\frac{H_{1m}^2}{\gamma_m-K_{w6m}}\right\}(K_{w7}+K_{w8})\cr
Q_2&=&\frac{(4d-2)(4d-4)}{\Gamma_{\phi_{\min}}^2}C_\infty^2\hat{c}_b\max_{m}\left\{2\Gamma^2+\frac{H_{1m}^2}{1-K_{w5m}}(K_{w7}+K_{w8})\right\}\cr
Q_3&=&\hat{D}_b.\nonumber
\end{eqnarray}
If $Q_1<1$, as is assumed in \Cref{a: ass8}, then
\begin{eqnarray}
\left.\frac{\partial \mathbf{Q}_{\Delta t}(y)}{\partial y^2}\right|_{y=0} = Q_1<1.\nonumber
\end{eqnarray}
By the strictly increasing derivative in $y$ there exists for every $\Delta t$ small enough a unique point $y^*$ such that $\left.\frac{\partial \mathbf{Q}_{\Delta t}(0,y)}{\partial y^2}\right|_{y=y^*}=:\mathbf{Q}'_{\Delta t}(0,y^*)=1$, which implies
$$Q_1 = 1-Q_2(y^*)^2\left(2+Q_3(y^*)^2\right)\exp(Q_3(y^*)^2)$$
and
$$\mathbf{Q}_{\Delta t}(0,y^*) = Q_0\Delta t+(y^*)^2-Q_2(y^*)^4\left(1+Q_3(y^*)^2\right)\exp(Q_3(y^*)^2).$$
Hence choose
$$H^* = \frac{Q_2}{Q_0}(y^*)^4\left(1+Q_3(y^*)^2\right)\exp(Q_3(y^*)^2),$$
such that every $\Delta t<H^*$ is small enough to satisfy \Cref{a: ass8}. Bolzano's theorem immediately gives the existence of an intersection point between $\mathbf{Q}_{H^*}(y)$ and $y^2$. Hence by the strictly positive derivative there exists a unique point $y^{**}\in(0,y^*)$ such that $\mathbf{Q}_{H^*}(y^{**})=(y^{**})^2$ and $\mathbf{Q}'_{H^*}(y^{**})<1$. Moreover by the strictly positive derivative larger than 1 there exists a unique point $y^{***}>y^*$ such that $\mathbf{Q}_{H^*}(y^{***})=(y^{***})^2$ and $\mathbf{Q}'_{H^*}(y^{***})>1$. Hence we have found the interval $(y^{**},y^{***})$ for which $Q_{H^*}(0,y)\leq \mathbf{Q}_{H^*}(y)<y^2$.
\end{proof}
To be more precise about the implications of $Q_1<1$ we need a new assumption.
\begin{assumption}
\label{a: ass9}
Assume
\begin{multline}
\frac{\left(d-\frac{1}{2}\right)(d+3)}{\Gamma^2_{\phi_{\min}}}\!\left[\sum_{m=1}^{d-1}\left(\frac{F_m^2}{1-\sum\limits_{j=1}^{d-1}\frac{\mathcal{M}_{j01m}}{2\sqrt{\gamma_j}}}+\frac{\eta_{\mathcal{M}vm2}}{2}\right)\right]\times\cr
\times\max_{1\leq m<d}\left\{\frac{H_{1m}^2}{\gamma_m-\frac{\mathcal{M}_{vm}^2}{2\eta_{\mathcal{M}vm2}}-\sum\limits_{j=1}^{d-1}\left(\frac{\mathcal{M}_{m01j}\sqrt{\gamma_m}}{2}+\frac{\mathcal{M}_{m11j}}{2}+\frac{\mathcal{M}_{j11m}}{2}\right)}\right\}<1.\nonumber
\end{multline}
while still satisfying
$$\min_m\left\{\gamma_m-\frac{\mathcal{M}_{vm}^2}{2\eta_{\mathcal{M}vm2}}-\sum\limits_{j=1}^{d-1}\left(\frac{\mathcal{M}_{m01j}\sqrt{\gamma_m}}{2}+\frac{\mathcal{M}_{m11j}}{2}+\frac{\mathcal{M}_{j11m}}{2}\right)\right\}>0.$$
\end{assumption}
The connection between \Cref{a: ass9} and \Cref{a: ass8} is now given by the next lemma.
\begin{lemma}\label{l: asscon}
Let \Cref{a: ass5} be satisfied, then \Cref{a: ass8,a: ass9} are equivalent.
\end{lemma}
\begin{proof}
\Cref{a: ass5} implies that the condition $1>K_{w5m}$ can be made to hold by choosing appropriate $\eta$-indices values. This enforces an upper bound on $\eta_{Lm2}>0$, which yields a lower bound for $1/\eta_{Lm2}$ and therefore a lower bound for $K_{w7}$. The constant $K_{w8}$ is only coupled to $K_{w5m}$, $K_{w6m}$ or $K_{w7}$ via $\eta_{\mathcal{M}vm2}$ in $K_{wm6}$. \Cref{a: ass5} implies that the condition $\gamma_m>K_{w6m}$ can be made to hold by choosing appropriate $\eta$-indices. However  $\eta_{\mathcal{M}vm2}$ is now not freely determinable. Hence, we need to keep in mind that $\gamma_m>K_{w6m}$ still needs to be valid. In the proof of \Cref{l: 1gron} it was shown that $\eta_{\mathcal{M}m01j1}$ equals $1/\sqrt{\gamma_m}$, while $\eta_{\mathcal{M}m11j1}$ equals $1$. A lower bound of $K_{w6m}$ is then given by
$$\frac{\mathcal{M}_{vm2}^2}{2\eta_{\mathcal{M}vm2}}+\sum_{j=1}^{d-1}\left(\frac{\mathcal{M}_{m01j}\sqrt{\gamma_m}}{2}+\frac{\mathcal{M}_{m11j}}{2}+\frac{\mathcal{M}_{j11m}}{2}\right).$$
A larger lower bound needs a value for freely determinable $\eta_{index}$. Therefore, the lower bounds for both $K_{w7}+K_{w8}$ and $K_{w6m}$ yield a necessary and sufficient condition for determining $\eta_{index}$ while still satisfying \Cref{a: ass8}.
\end{proof}
We introduce a set of new constants.
\begin{definition}\label{d: def2}
 Use the notation of the proof of \Cref{l: 1Vbound}. Introduce
$$
\tilde{y}:=\sqrt{\frac{1}{\hat{D}_b}W_0\left(\min_{l}\left\{\frac{\left(\frac{1}{2}\left(\frac{1-\phi_{\min}}{C_\infty}\right)^2-\sum\limits_{j\neq l}\phi_{l0}^2\right)\hat{D}_b}{(d-1-\delta_{d-1,l})\hat{Z}_a+(1+(d-1)\delta_{d-1,l})\hat{c}_b+\hat{\mathcal{P}}\delta_{d-1,l}}\right\}\right)}$$
and
$$\tilde{z} := \frac{2}{3}W_0\left(\frac{3Q_3}{4Q_2}(1-Q_1)\right),$$
where
\begin{eqnarray}
\hat{Z}_a&:=&K_{a\phi1}+\frac{\max\limits_{m}\left\{\frac{K_{a\phi6m}}{1-K_{w5m}},\frac{K_{a\phi7m}}{\gamma_m-K_{w6m}}\right\}}{\min\limits_{m}\{1-K_{w5m},\gamma_m-K_{w6m}\}}(K_{w7}+K_{w8}),\cr
\hat{\mathcal{P}}&:=&5+2d(K_{w7}+K_{w8}),\nonumber
\end{eqnarray}
 and $W_0(\cdot)$ denotes the standard product log branch, the inverse of $x\exp(x)$, through the origin.
\end{definition}
 The value $y^*$ might not be expressible in standard functions preventing any explicit calculation of $H^*$. We can however determine another upper bound for $\Delta t$, which can be calculated explicitly.
\begin{lemma}
\label{l: tildeZ}
Let $\tilde{z}$ be as in \Cref{d: def2}. Then $\tilde{z}<Q_3(y^*)^2$ holds and the identity
$$H^{**}=\frac{Q_2}{Q_0Q_3^2}\tilde{z}^2\left(1+\tilde{z}\right)\exp(\tilde{z})\qquad\text{implies}$$
$$H^{**}<H^*=\frac{Q_2}{Q_0}(y^*)^4\left(1+Q_3(y^*)^2\right)\exp(Q_3(y^*)^2).$$
\end{lemma}
\begin{proof}
Let $X = Q_3(y^*)^2$, then the identity for $H^*$ in \Cref{l: 1Vbound} becomes\\ $H^*=\frac{Q_2}{Q_0Q_3^2}X^2\left(1+X\right)\exp(X)$
with $X>0$ satisfying
$$\frac{1-Q_1}{Q_2}Q_3=X(2+X)\exp(X)<\frac{4}{3}\cdot\frac{3}{2}X\exp\left(\frac{3}{2}X\right).$$
Introduce $\tilde{z}<X$ as $\frac{1-Q_1}{Q_2}Q_3=\frac{4}{3}\cdot\frac{3}{2}\tilde{z}\exp\left(\frac{3}{2}\tilde{z}\right),$
then $\tilde{z} = \frac{2}{3}W_0\left(\frac{3}{4}\frac{1-Q_1}{Q_2}Q_3\right)$. Thus the explicitly determinable upper bound
$H^{**}=\frac{Q_2}{Q_0Q_3^2}\tilde{z}^2\left(1+\tilde{z}\right)\exp(\tilde{z})$
implies $H^{**}<H^*$.
\end{proof}
 The previous lemma showed another upper bound constraint for the time step $\Delta t$. Yet another constraint is necessary to obtain the intersection between $S$ and $\mathcal{R}$, leading to the following assumption.
\begin{assumption}\label{a: ass6} Let $Q_0,Q_1,Q_2$ and $Q_3$ be as in the proof of \Cref{l: 1Vbound}. Let $B$\\ \vphantom{a}\qquad\quad be given by \Cref{d: def1}, let $K_{a\phi2l}$ be given by \Cref{eq: xkphi2}, and let $\tilde{z}$ and \\ \vphantom{a}\qquad\quad $\tilde{y}$ be given by \Cref{d: def2}. Assume
$$H<\min\left\{\frac{(1-Q_1)^2}{4Q_2Q_0},\frac{\tilde{y}^4Q_2}{Q_0},\frac{Q_2\tilde{z}^2}{Q_0Q_3^2}(1+\tilde{z})\exp(\tilde{z}),\frac{0.6838}{B},\frac{0.6838}{D_a}\right\}.$$
\end{assumption}
The intersection between $S$ and $\mathcal{R}$ must yield the admissible values for $T_{\Delta t}-t_0$ and $V$ for which \Cref{a: ass3} is satisfied. However it is not yet clear whether such an intersection exists. This issue is addressed in the next lemma.
\begin{lemma}
\label{l: 1intersect}
The intersection $S\cap\mathcal{R}$ is nonempty, if \Cref{a: ass6,a: ass7,a: ass8} hold with the constants $\tilde{y}$, $\tilde{z}$, $\hat{Z}_a$, $\hat{c}_b$, $\hat{D}_b$ and $\hat{\mathcal{P}}$ of \Cref{d: def2}.
\end{lemma}
\begin{proof}
We have already proven in \Cref{l: tint} that independent of $\Delta t$ there exists an interval $(0,\hat{y})$ such that $P_\alpha(0,y)<\frac{1}{2C_\infty^2}(1-\phi_{min})^2$ for all $y\in(0,\hat{y})$ and all $\alpha\in\mathbf{P}_d$. However, we have not determined the explicit value of $\hat{y}$. Furthermore, we have already proven there exists an interval $(y^{**},y^{***})$ with $y^{**}>0$ in which that $Q_{H^*}(y)<y^2$. Let $Q_{\Delta t}(0,y)\leq \mathbf{Q}_{\Delta t}(y)=Q_0\Delta t+Q_1y^2+Q_2y^4\exp(Q_3y^2)<y^2$ for all $y\in(y_{\Delta t}^{**},y_{\Delta t}^{***})$ and for all $\Delta t\in(0,H^{**})$. The intersection $S\cap\mathcal{R}$ can then be proven to be nonempty if $y_{\Delta t}^{**}<\hat{y}$ for $\Delta t<H^{**}$ small enough. Thus we need to determine an upper bound for $y^{**}_{\Delta t}$ and a lower bound for $\hat{y}$. We know from \Cref{l: 1Vbound} that $y_{\Delta t}^{**}$ satisfies $Q_0\Delta t+Q_1(y_{\Delta t}^{**})^2+Q_2(y_{\Delta t}^{**})^4\exp\left(Q_3(y_{\Delta t}^{**})^2\right)=(y_{\Delta t}^{**})^2$. By replacing $\exp(Q_3y^2)$ in $\mathbf{Q}_{\Delta t}(y)$ by 1 we obtain a function with a less increasing derivative. Hence we obtain an upper bound for $y_{\Delta t}^{**}$ by calculating $Q_0\Delta t+Q_1y^2+Q_2y^4=y^2$, which yields
\begin{equation}
y^2 =\frac{1-Q_1}{2Q_2}-\sqrt{\frac{(1-Q_1)^2}{4Q_2^2}-\frac{Q_0\Delta t}{Q_2}}\leq\sqrt{\frac{Q_0\Delta t}{Q_2}}.\nonumber
\end{equation}
if
\begin{equation}
\Delta t \leq \frac{(1-Q_1)^2}{4Q_2Q_0}\nonumber
\end{equation}
Introduce $y_{\Delta t}$ as an upper bound to $y^{**}_{\Delta t}$, then we can choose
\begin{equation}
y_{\Delta t}=\sqrt[4]{\frac{Q_0\Delta t}{Q_2}}.\nonumber
\end{equation}
The upper bound of $\tilde{y}$ requires first a detailed description of $\mathcal{P}(0,y)$. After applying Cauchy-Schwartz inequality and Cauchy's inequality to \Cref{eq: phid} and inserting \Cref{c: 1gron,c: 1gron2,c: 1gron3} we observe that all terms with $\|w_m^k||_{H^1}$ and $\|\phi_n^k||_{H^1}$ for $n\neq d-1$ will yield an upper bound of the order $T_{\Delta t}-t_0= x\downarrow0$. Furthermore all terms $\|\phi_{d-1}^k||_{L^2}$ without prefactors containing $\|\mathcal{D}_{\Delta t}^k(w_m)||_{H^1}$ or $\|\mathcal{D}_{\Delta t}^k(\phi_l)||_{L^2}$ will yield an exponential with and exponent containing the prefactor $T_{\Delta t}-t_0= x\downarrow0$. Hence only the terms $\|v^k\|_{H^1}$, $\|\mathcal{D}_{\Delta t}^k(w_m)||_{H^1}$ or $\|\mathcal{D}_{\Delta t}^k(\phi_l)||_{L^2}$ with possibly a prefactor $\|\phi_{d-1}^k||_{L^2}$, which can be set equal to 1, will lead to upper bounds with factors $V^2= y^2$. Therefore we obtain
\begin{eqnarray}
\|\phi_{d-1}^k\|_{L^2}^2(0,y)\leq\left(2C_{\infty}^2+1+2d(K_{w7}+K_{w8})\right)y^2=:\hat{\mathcal{P}}y^2.\nonumber
\end{eqnarray}
Moreover, \Cref{eq: phidz} with \Cref{c: 1gron2} yields
\begin{eqnarray}
\|\partial_z\phi_{d-1}^k\|_{L^2}^2(0,y)\leq(d-1)\sum_{l\neq d-1}\|\partial_z\phi_{l}^k\|_{L^2}^2(0,y)\leq (d-1)\hat{c}_by^2\exp(\hat{D}_by^2).\nonumber
\end{eqnarray}
Hence, we have
\begin{eqnarray}
\mathcal{P}(0,y)=\hat{\mathcal{P}}y^2+(d-1)\hat{c}_by^2\exp(\hat{D}_by^2).\nonumber
\end{eqnarray}
Similarly, by inserting the result of \Cref{c: 1gron2,c: 1gron3}, we obtain
\begin{eqnarray}
\mathcal{P}_\alpha(0,y)=\begin{dcases}
\hat{\mathcal{P}}y^2+d\hat{c}_by^2\exp(\hat{D}_by^2)+\sum_{j=1}^{d-1}\phi_{\alpha_j0}^2+(d-2)\hat{Z}_ay^2&\text{ if }d-1\neq\alpha_d\cr
\hat{c}_by^2\exp(\hat{D}_by^2)+\sum_{l\neq d-1}\phi_{l0}^2+(d-1)\hat{Z}_ay^2&\text{ if }d-1=\alpha_d
\end{dcases}\nonumber
\end{eqnarray}
Now, take $y^2\exp(\hat{D}_by^2)$ as an upper bound of $y^2$ to obtain a lower bound $\tilde{y}$ of $\hat{y}$. This yields immediately the value of $\tilde{y}$ as stated in \Cref{d: def2}. Thus for $\Delta t$ smaller than $\tilde{y}^4Q_2/Q_0$ we observe an intersection of $S\cap\mathcal{R}$ if $S$ exists. The existence of $S$, as shown in \Cref{l: 1Vbound}, gives another upper bound for $\Delta t$ and therefore one must take the minimum of the two.
\end{proof}
\begin{theorem}[Existence of a weak solution to the discretised system with $\Delta t$ independent bounds]
\label{t: timedomain}
Let \Cref{a: ass1,a: ass4,a: ass5,a: ass6,a: ass7,a: ass9} hold. Then there exists $T-t_0,V\in S\cap\mathcal{R}$ independent of $\Delta t$ such that there exists a solution $(\phi^k,v,w^k)$ of the discretised system satisfying
\begin{subequations}
\begin{align}
\sum_{j=0}^k\left\|v^j\right\|_{L^2}^2\Delta t&\leq V^2\cr
\sum_{j=0}^k\left\|\partial_z v^j\right\|_{L^2}^2\Delta t&\leq V^2\cr
\phi^k&\in[\phi_{\min},1-(d-1)\phi_{\min}]^d\cr
\left\|\phi_1^k\right\|_{H^1},\ldots,\left\|\phi_d^k\right\|_{H^1}&\leq \mathcal{C}\cr
\sum_{j=1}^k\left\|\phi_1^j\right\|_{H^2}^2\Delta t,\ldots,\sum_{j=1}^k\left\|\phi_d^j\right\|_{H^2}^2\Delta t&\leq\mathcal{C}\cr
\sum_{j=1}^k\left\|\mathcal{D}_{\Delta t}^k(\phi_1^j)\right\|_{L^2}^2\Delta t,\ldots, \sum_{j=1}^k\left\|\mathcal{D}_{\Delta t}^k(\phi_d^j)\right\|_{L^2}^2\Delta t&\leq\mathcal{C}\cr
\left\|w_1^k\right\|_{H^2},\ldots,\left\|w_{d-1}^k\right\|_{H^2}&\leq \mathcal{C}\cr
\sum_{j=1}^k\left\|\mathcal{D}_{\Delta t}^k(w_1^j)\right\|_{H^1}^2\Delta t,\ldots,\sum_{j=1}^k\left\|\mathcal{D}_{\Delta t}^k(w_{d-1}^j)\right\|_{H^1}^2\Delta t&\leq\mathcal{C}\nonumber
\end{align}
\end{subequations}
for all $t_0\leq t_k\leq T\leq T_{\Delta t}$ and for all $\Delta t\in(0,H)$ with $\mathcal{C}>0$ independent of $\Delta t$.
\end{theorem}
\begin{proof} The existence of $T-t_0$ and $V$ for which \Cref{a: ass3} holds, has been shown in \Cref{l: 1intersect}. Furthermore, the strictly increasing derivatives of both $\mathcal{P}_\alpha$ and $\mathcal{Q}_{\Delta t}$ with respect to $\Delta t$ show that elements of the region $S\cap\mathcal{R}$ can always be chosen for $\Delta t<H$. The $\Delta t$ independent bounds is a consequence of \Cref{t: 1dim} and \Cref{a: ass3} in combination with the bounds obtained in \Cref{c: 1gron,c: 1gron2,c: 1gron3}. The $H^2$ norms follow directly from \Cref{t: 1dim} in \Cref{s: app0} applied to \Cref{eq: disc1,eq: disc3}.
\end{proof}
\section{Interpolation functions and their time continuous limit}$\;$\\
\label{s: sec6}
In this section we will construct interpolation function on $(t_0,T)\times(0,1)$ for our variables $(\phi,v,w)$, and investigate their limits for $\Delta t\downarrow0$.\\
\Cref{t: timedomain} shows that there exists a constant $\mathcal{C}$ and Sobolev spaces $X_j,Y_j,Z_j$ such that \begin{eqnarray}
\sup_{0\leq k\leq K}\left\|u^k\right\|_{X_j}^2\leq\mathcal{C}&<&\infty\cr
\sum_{k=0}^K\left\|u^k\right\|_{Y_j}^2\Delta t\leq\mathcal{C}&<&\infty\cr
\sum_{k=1}^K\left\|\mathcal{D}_{\Delta t}^k(u)\right\|_{Z_j}^2\Delta t\leq\mathcal{C}&<&\infty\nonumber
\end{eqnarray}
for all $u\in(\phi,v,w)$ with $K\Delta t = T_{\Delta t}-t_0$.\\
These bounds guarantee that interpolation functions $\hat{u}(t) := u^{k-1}+(t-t_{k-1})\mathcal{D}_{\Delta t}^k(u)$ and $\overline{u}(t) := u^{k}$ lie in the desired Bochner spaces.
\begin{lemma}
\label{l: 1regul}
Let $H$ be given by \Cref{a: ass6}, let $0<\Delta t<H$  be fixed, $\mathcal{C}>0$ and let $X$ be a Sobolev space. The following implications hold:
\begin{subequations}
\begin{align}
\sup_{0\leq k\leq K}\left\|u^k\right\|_{X}^2\leq\mathcal{C}<\infty&\Rightarrow\begin{dcases}
\hat{u}_{\Delta t}\in L^\infty(t_0,T_{\Delta t};X)\cr
\overline{u}_{\Delta t}\in L^\infty(t_0,T_{\Delta t};X)\cr
\end{dcases}\cr
\sum_{k=0}^K\left\|u^k\right\|_{X}^2\Delta t\leq\mathcal{C}<\infty&\Rightarrow\begin{dcases}\hat{u}_{\Delta t}\in L^2(t_0,T_{\Delta t};X)\cr
\overline{u}_{\Delta t}\in L^2(t_0,T_{\Delta t};X)
\end{dcases}\cr
\sum_{k=1}^K\left\|\mathcal{D}_{\Delta t}^k(u)\right\|_{X}^2\Delta t\leq\mathcal{C}<\infty&\Rightarrow\begin{dcases}
\frac{\partial\hat{u}_{\Delta t}}{\partial t}&\!\!\!\in L^2(t_0,T_{\Delta t};X)\cr
\frac{\overline{u}_{\Delta t}(t)-\overline{u}_{\Delta t}(t-\Delta t)}{\Delta t}&\!\!\!\in L^2(t_0+\Delta t,T_{\Delta t};X)
\end{dcases}\nonumber
\end{align}
\end{subequations}
\end{lemma}
\begin{proof}
The measurability of $\overline{u}_{\Delta t}$ and $\hat{u}_{\Delta t}$ is easily established since the piecewise constant functions are measurable and dense in the set of piecewise linear functions. Furthermore, the construction of $\overline{u}_{\Delta t}$ and $\hat{u}_{\Delta t}$ yields
\begin{subequations}
\begin{align}
\int_{t_0}^{T_{\Delta t}}\|\hat{u}_{\Delta t}\|_X^2(t)\mathrm{d}t = \sum_{k=1}^K\int_{t_{k-1}}^{t_k}\|\hat{u}_{\Delta t}\|_X^2(t)\mathrm{d}t&\leq 2\sum_{k=0}^K\|u^k\|_{X}^2\Delta t,\cr
\int_{t_0}^{T_{\Delta t}}\|\overline{u}_{\Delta t}\|_X^2(t)\mathrm{d}t = \sum_{k=1}^K\int_{t_{k-1}}^{t_k}\|\overline{u}_{\Delta t}\|_X^2(t)\mathrm{d}t&\leq \sum_{k=1}^K\|u^k\|_{X}^2\Delta t,\cr
\text{and}\quad\esssup_{t\in(t_0,T_{\Delta t})}\|\hat{u}_{\Delta t}\|_X(t)=\esssup_{t\in(t_0,T_{\Delta t})}\|\overline{u}_{\Delta t}\|_X(t) &= \sup_{t_k\in[t_0,T_{\Delta t}]}\|u^k\|_{X}.\nonumber
\end{align}
\end{subequations}
From the definition of derivative, it is easily seen that $\hat{u}_{\Delta t}$ has a strong derivative for a.e. $t\in(0,T_{\Delta t}]$, since we have
\begin{equation}
\lim_{h\downarrow0}\left\|\frac{\hat{u}_{\Delta t}(t+h)-\hat{u}_{\Delta t}(t)}{h}-\frac{u^k-u^{k-1}}{\Delta t}\right\|=0\nonumber
\end{equation}
for $t\in(t_{k-1},t_k)$ with $t_0<t_k\leq T_{\Delta t}$. Hence we obtain
\begin{subequations}
\begin{align}
\int_{t_0}^{T_{\Delta t}}\left\|\partial_t\hat{u}_{\Delta t}\right\|_X^2(t)\mathrm{d}t &= \sum_{k=1}^K\int_{t_{k-1}}^{t_k}\left\|\partial_t\hat{u}_{\Delta t}\right\|_X^2(t)\mathrm{d}t\cr
&\leq \sum_{k=1}^K\left\|\mathcal{D}_{\Delta t}^k(u)\right\|_{X}^2\Delta t.\cr
\int_{t_0+\Delta t}^{T_{\Delta t}}\left\|\frac{\overline{u}_{\Delta t}(t)-\overline{u}_{\Delta t}(t-\Delta t)}{\Delta t}\right\|_X^2\mathrm{d}t &= \sum_{k=1}^K\int_{t_{k-1}}^{t_k}\left\|\frac{\overline{u}_{\Delta t}(t)-\overline{u}_{\Delta t}(t-\Delta t)}{\Delta t}\right\|_X^2\mathrm{d}t\cr
&\leq \sum_{k=1}^K\left\|\mathcal{D}_{\Delta t}^k(u)\right\|_{X}^2\Delta t.\nonumber
\end{align}
\end{subequations}
\end{proof}
With the bounded norms, independent of $\Delta t$, and the $\Delta t$-independent time-interval $[t_0,T]$ with $T\leq T_{\Delta t}$ the weak convergence in $\Delta t$ of functions $\hat{u}_{\Delta t}$ and $\overline{u}_{\Delta t}$ defined on $(t_0,T)\times(0,1)$ is guaranteed by the Eberlein-Smulian theorem, as stated in \cite{CaVyMo2007}.\\
$\;$\\
Weak convergence of products of functions defined on $(t_0,T)\times(0,1)$ is guaranteed by strong convergence of all but one function in the product. The strong convergence is given by the Lions-Aubin-Simon lemma, originally stated in \cite{Simon1986}. We use a version of \cite{DrJu2012} with slight modifications as stated in \cite{ChJuLi2013}.

\begin{lemma}[Lions-Aubin-Simon]
\label{l: LAS}
Let $X$, $B$, and $Y$ be Banach spaces such that the embedding $X \hookrightarrow B$ is
compact and the embedding $B \hookrightarrow Y$ is continuous. Furthermore, let either $1 \leq p < \infty$,
$r = 1$ or $p = \infty$, $r > 1$, and let $(u_\tau)$ be a sequence of functions, that are constant on each
subinterval $(t_{k-1}, t_{k})$, satisfying
\begin{subequations} \begin{align}
\left\|u_\tau\right\|_{L^p(0,T;X)}&\leq\,C_0,\label{eq: LASbound1}\\
\left\|u_{\tau}(t)-u_{\tau}(t-\tau)\right\|_{L^r(t_0+\tau,T;Y)}&\leq\,C_0\tau^\alpha\label{eq: LASbound2}
 \end{align}\end{subequations}
for $\alpha = 1$ and for all $\tau>0$, where $C_0 > 0$ is a constant that is independent of $\tau$. If $p < \infty$, then $(u_\tau )$ is relatively compact in $L^p(0,T;B)$. If $p = \infty$, then there exists a subsequence of $(u_\tau )$ that converges in
each space $L^q(0,T;B)$, $1 \leq q < \infty$, to a limit that belongs to $C^0([0, T];B)$.\\
Moreover we cannot replace $\alpha = 1$ with $\alpha \in (0,1)$.
\end{lemma}

Compactness results related to $H^r(\Omega)$ with $r\geq0$ can be found in \cite{LiMa1}. For one-dimensional bounded $\Omega$ compactness results can be found in section 8.2 of \cite{Brezis2010}.

\begin{theorem}
\label{t: comp}
Let $s\in\mathbf{R}$. If $\Omega\subset\mathbf{R}^n$ is bounded and has a $(n-1)$-dimensional infinitely differentiable boundary $\Gamma$ with $\Omega$ being locally on one side of $\Gamma$, then the injection $H^s(\Omega)\hookrightarrow H^{s-\epsilon}(\Omega)$ is compact for every $\epsilon>0$.\\
Let $\Omega\subset\mathbf{R}$ be bounded, then for all $m>0$ integer and $p\in(0,\infty]$ we have
\begin{equation}
\label{eq: compC}
W^{m+1,p}(\Omega)\hookrightarrow\hookrightarrow C^{m}(\overline{\Omega})\hookrightarrow W^{m,p}(\Omega).
\end{equation}
\end{theorem}
We conclude that there exists a subsequence $(\Delta t)\downarrow0$ for which we have both weak and strong convergence (but in different functions spaces) of both $\hat{u}_{\Delta t}$ and $\overline{u}_{\Delta t}$. A priori these limits $\hat{u}$ and $\overline{u}$ are not necessarily the same, however with the strong convergence we show that the limits are identical.
\begin{lemma}\label{l: 1limits}
Let $\overline{u}_{\Delta t}\rightarrow\overline{u}$ strongly in $L^2(0,T;X)$, $\hat{u}_{\Delta t}\rightharpoonup\hat{u}$ weakly in $L^2(0,T;X)$ and let $\sum_{l=1}^K\|\mathcal{D}_{\Delta t}^l(u)\|_X^2\Delta t\leq \mathcal{C}<\infty$ with $\mathcal{C}$ independent of $\Delta t$, then $\hat{u}=\overline{u}$.
\end{lemma}
\begin{proof}
Based on strong convergence, we have $\int_{t_0}^T\|\overline{u}_{\Delta t}(t)-\overline{u}(t)\|_X^2\mathrm{d}t\rightarrow0$ as $\Delta t\downarrow0$. From the construction of both $\hat{u}_{\Delta t}$ and $\overline{u}_{\Delta t}$ we have $\overline{u}_{\Delta t}(t)-\hat{u}_{\Delta t}(t) = \frac{t_k-t}{\Delta t}(u^k-u^{k-1})$, from which we obtain
\begin{equation}
\int_{t_0}^{T_{\Delta t}}\|\overline{u}_{\Delta t}(t)-\hat{u}_{\Delta t}(t)\|_X^2\mathrm{d}t\leq\sum_{k=1}^K\|\mathcal{D}_{\Delta t}^k(u)\|_X^2(\Delta t)^2\leq \mathcal{C}\Delta t\downarrow0\nonumber
\end{equation}
Thus by the triangle inequality we obtain $\int_{t_0}^T\|\hat{u}_{\Delta t}(t)-\overline{u}(t)\|_X^2\mathrm{d}t\rightarrow0$. Hence, $\hat{u}_{\Delta t}\rightarrow\overline{u}$ strongly in $L^2(t_0,T;X)$ as $\Delta t\downarrow0$. Now take an arbitrary $\psi\in L^2(t_0,T;X)$, then we have
\begin{equation}
\left|\int_0^T\int_0^1(\hat{u}-\overline{u})\psi\mathrm{d}x\mathrm{d}t\right|\leq \left|\int_0^T\int_0^1(\hat{u}_{\Delta t}-\hat{u})\psi\mathrm{d}x\mathrm{d}t\right|
+\left|\int_0^T\int_0^1(\hat{u}_{\Delta t}-\overline{u})\psi\mathrm{d}x\mathrm{d}t\right|\rightarrow0\nonumber
\end{equation}
by the weak convergence of $\hat{u}_{\Delta t}$ to $\hat{u}$ and the strong convergence of $\hat{u}_{\Delta t}$ to $\overline{u}$. We have chosen $\psi$ arbitrarily in $L^2(t_0,T;X)$. Hence, $\hat{u}=\overline{u}$ in $L^2(t_0,T;X)$.
\end{proof}
At this point, we have shown the existence of a weak and strong limit to the discrete functions $u^k_l$ on the interior of $[t_0,T]\times[0,1]$. Furthermore, we have initial conditions given by \Cref{eq: init}, while the final conditions at time $t=T$ exist by the construction of the interior functions and the determination of the interval $[t_0,T]$. We have not yet shown that the limit function on $(t_0,T)\times(0,1)$ has boundary values on the lateral boundary $(t_0,T)\times\{0,1\}$ that satisfy the boundary conditions of the continuous system. First we show that we can apply the trace theorem to identify unique boundary values for functions defined on the interior, and that these trace functions on the boundary satisfy compatibility relations. For convenience we follow the notation of Bochner spaces as stated in Lions and Magenes \cite{LiMa2} by introducing $Y(I,X)$ for normed spaces $X,Y(I)$ of interval $I$ as the space of functions $u(t)\in X$ satisfying $\left\|\|u\|_X\right\|_{Y(I)}<\infty$, and by introducing the spaces
 \begin{equation}
 H^{r,s}(I\times\Omega):=L^2(I,H^r(\Omega))\cap H^s(I,L^2(\Omega))
 \end{equation}
for $r,s>0$ where we will use the notation $Q := I\times\Omega$ with $\Omega\subset\mathbf{R}^n$ for the joint domain and $\Sigma = I\times\Gamma$ with $\Gamma:=\partial\Omega$ for the lateral boundary. Bochner space theory, as found in \cite{LiMa2}, shows the existence of trace functions and global compatibility relations, which we summarized in \Cref{t: trace}.

\begin{theorem}\label{t: trace}
Let $u\in H^{r,s}(Q)$ with $r,s\geq0$.\\
If $0\leq j <r-\frac{1}{2}$ integer, then $\partial^j_\nu u\in H^{\mu_j,\nu_j}(\Sigma)$, where $\partial_\nu^j$ is the $j$th order normal derivative on $\Sigma$, oriented toward the interior of $Q$; $\frac{\mu_j}{r}=\frac{\nu_j}{s} = \frac{r-j-\frac{1}{2}}{r}$ and where $u\rightarrow\partial_\nu^ju$ are continuous linear mappings of $H^{r,s}(Q)\rightarrow H^{\mu_j,\nu_j}(\Sigma)$.

Introduce functions $f_i(x)=\partial_t^iu(x,0)$ and $g_j(\tilde{x},t)=\partial_x^ju((x',0),t)$ for\\ $(x',0),x\in\mathbf{R}^n$ for some local indexation of the coordinates on $\Gamma$.\\
Introduce the product space $F$ of elements
 \begin{equation}
 \{f_i,g_j\}\in F = \prod_{i<s-\frac{1}{2}}H^{p_i}(\Omega)\times \prod_{j<r-\frac{1}{2}}H^{\mu_j,\nu_j}(\Sigma)
  \end{equation}
with $\frac{p_i}{r} = \frac{s-i-\frac{1}{2}}{s}$ and $\frac{\mu_j}{r} = \frac{\nu_j}{s} = \frac{r-j-\frac{1}{2}}{r}$.\\
Let $F_0$ be the vector subspace of $F$ which satisfies $\partial_t^ig_j(x',0) = \partial_{x_n}^jf_i(x',0)$ and\\ $\int_0^T\int_{\Gamma}\left|\partial_t^ig_j(x',\sigma^r) - \partial_{x_n}^jf_i(x',\sigma^s)\right|^2\text{d}x'\frac{\text{d}\sigma}{\sigma}<\infty$.\\
If $r,s>0$ and $0\leq\frac{j}{r}+\frac{i}{s}<1-\frac{1}{2}\left(\frac{1}{r}+\frac{1}{s}\right)$ integers, then the map $u\mapsto\{f_i,g_j\}$ is a continuous linear surjection of $H^{r,s}(Q)\rightarrow F_0$.
\end{theorem}
Together with the Lions-Aubin-Simon lemma, which states that the strong limit is an element of $L^2(t_0,T;C^j(0,1))$ for some integer $j$, we can show that the trace functions satisfy the boundary conditions, essentially because the trace functions are a limit of the trace functions of the interpolation functions for which the boundary conditions do apply.
\begin{lemma}
\label{l: bound}
Let $a\in\{0,1\}$. Let $u_{\Delta t}$ be either $\hat{u}_{\Delta t}$ or $\overline{u}_{\Delta t}$. Let $u_{\Delta t}|_{z=a}$ be the single-sided trace of $u_{\Delta t}$. For $r,s,j,i,s-i\geq0$ and $r-j-1/2>0$ we have
\begin{subequations}
\begin{align}
u_{\Delta t}\in H^{r,s}(Q)&\Rightarrow \left.\partial_\nu^j\partial_t^i u_{\Delta t}\right|_{z=a}\in H^{r-j-\frac{1}{2},\frac{r-j-1/2}{r}(s-i)}(\Sigma) = H^{\frac{r-j-1/2}{r}(s-i)}(t_0,T)\nonumber
\end{align}
\end{subequations}
with weak limit $u|_{z=a}$ in $H^{\frac{r-j-1/2}{r}(s-i)}(t_0,T)$. Additionally, if $r\geq j+1/2$, $s\geq i+1/2$ then there exists a subsequence $(\Delta t)$ such that $\left.\partial_\nu^j\partial_t^i u_{\Delta t}\right|_{z=a}$ converges weakly in $L^2(t_0,T)$ to $\left.\partial_\nu^j\partial^i_t u\right|_{z=a}=\left.\partial_t^i\partial_\nu^ju\right|_{z=a}\in L^2(t_0,T)$. Moreover, if $r\geq j+1$, $s\geq i+1$, then there exists a subsequence $(\Delta t)$ such that $\left.\partial_\nu^j\partial_t^i u_{\Delta t}\right|_{z=a}$ converges strongly in $L^2(t_0,T)$ to $\left.\partial_\nu^j\partial^i_t u\right|_{z=a}=\left.\partial_t^i\partial_\nu^ju\right|_{z=a}\in C^0(t_0,T)$.\\
If on the other hand $\left.\partial_\nu^j\partial_t^i u_{\Delta t}\right|_{z=a}=C\in\mathbf{R}$, then it converge strongly to\textcolor{white}{blablabla} $\left.\partial_\nu^j\partial_t^i u\right|_{z=a}=C$ regardless of the function space on $[t_0,T]$ or the sequence $(\Delta t)$.
\end{lemma}
\begin{proof}
\Cref{t: trace} immediately gives the appropriate spaces for $u_{\Delta t}|_{z=a}$. Furthermore, \Cref{l: LAS} shows that weak convergence in $H^{j+1,i+1}(Q)$ implies the existence of a subsequence $(\Delta t)$ with weak convergence of $\partial_z^j\partial_t^i u_{\Delta t}$ in $L^q(t_0,T;C^0([0,1]))$ for each $1\leq q<\infty$ to $\partial_z^j\partial_t^i u\in C^0(\overline{Q})$ and strong convergence of $u_{\Delta t}$ to $u$ in $H^i(t_0,T;C^j([0,1]))$. The continuity $C^i([t_0,T];C^j([0,1]))$ shows that the trace operator is both an evaluation $z=a$ and a limit from $z\in(0,1)$ to $z=a\in\{0,1\}$. Additionally, the continuity shows that the derivatives commutate. The strong convergence on the boundary follows from the following inequalities:
\begin{subequations}
\begin{align}
\lim_{\Delta t\downarrow0}\left\|\left.\partial^j_\nu\partial^i_tu_{\Delta t}\right|_{z=a}\!\!\!\!\!\!-\left.\partial^j_\nu\partial^i_tu\right|_{z=a}\right\|_{L^2(t_0,T)}&\leq\lim_{\Delta t\downarrow0}\left\|\sup_{\overline{z}\in[0,1]}\!\!\!\left|\partial^j_z\partial^i_tu_{\Delta t}(\overline{z})\!-\!\partial^j_z\partial^i_tu(\overline{z})\right|\right\|_{L^2(t_0,T)}\cr
&=\lim_{\Delta t\downarrow0}\left\|\partial^j_z\partial^i_tu_{\Delta t}-\partial^j_z\partial^i_tu\right\|_{L^2(t_0,T;C^0([0,1]))}\cr
&\leq\lim_{\Delta t\downarrow0}\left\|u_{\Delta t}\!-\!u\right\|_{H^j(t_0,T;C^i([0,1]))}=0\nonumber
\end{align}
\end{subequations}
The strong convergence of $\partial_\nu^j\partial_t^iu_{\Delta t}$ in $L^2(t_0,T)$ states that the weak convergence of $\partial_\nu^j\partial_t^iu_{\Delta t}$ in $L^2(t_0,T)$, in particular the one induced by the weak convergence of $\partial_z^j\partial_t^i u_{\Delta t}$ in $L^q(t_0,T;C^0([0,1]))$ with $q=2$, has the same limit as the strong convergence. Hence, the strong limit must be in $C^0(t_0,T)$.\\
The weak convergence of $\left.\partial_\nu^j\partial_t^i u_{\Delta t}\right|_{z=a}\in L^2(t_0,T)$ to $\left.\partial_\nu^j\partial^i_t u\right|_{z=a}=\left.\partial_t^i\partial_\nu^ju\right|_{z=a}\in L^2(t_0,T)$ can be easily seen. The weak convergence itself is a direct application of \Cref{l: LAS,t: trace,t: comp}, while the commutating derivatives are a consequence of the weak derivative structure itself.\\
The strong convergence for the case $\left.\partial_\nu^j\partial_t^i u_{\Delta t}\right|_{z=a}=C\in\mathbf{R}$ is trivial.
\end{proof}
We will now show the weak convergence of the semi-discrete approximations to a weak solution of the continuous system.
\begin{theorem}
\label{t: convergence}
If the conditions of \Cref{t: timedomain} are satisfied, then there exist constants $T-t_0>0$ and $V>0$ such that there exist functions:
\begin{subequations}
\begin{align}
\phi_l&\in L^2(t_0,T;H^2([0,1]))\cap L^\infty(t_0,T;H^1(0,1))\cap C^0([t_0,T];C^0([0,1]))\cr
&\qquad\cap H^1(t_0,T;L^2(0,1))\cr
v&\in L^2(t_0,T;H^1(0,1))\cr
w_m&\in L^\infty(t_0,T;H^2(0,1))\cap C^0([t_0,T];C^1([0,1]))\cap H^1(t_0,T;H^1(0,1))\cr
\mathcal{W}&\in H^1(t_0,T)\nonumber
\end{align}
\end{subequations}
for all $l\in\{1,\ldots,d\}$ and $m\in\{1,\ldots,d-1\}$ satisfying \Cref{a: ass3}, and \\
\begin{subequations}
\begin{align}
\left(\partial_t\phi_l+I_l(\phi)\partial_z(\Gamma(\phi)v)+\sum_{m=1}^{d-1}\sum_{i,j=0}^1\partial_z^i(B_{lijm}(\phi)\partial_t^kw_m)-G_{\phi,l}(\phi),\psi_{1l}\right)_Q\label{eq: 1weak}\\
\qquad=-\left(\delta_l\partial_z\phi_l,\partial_z\psi_{1l}\right)_Q,\cr
\left(\partial_z(\Gamma(\phi)v)+\sum_{m=1}^{d-1}\sum_{j=0}^1\partial_z(H_{jm}(\phi)\partial_t^kw_m)-G_v(\phi),\psi_2\right)_Q\label{eq: 2weak}\\
\qquad=0,\cr
\left(\partial_tw_m+F_m(\phi)w_m-G_{w,m}(\phi),\psi_{3m}\right)_Q\label{eq: 3weak}\\
+\left(\sum_{j=1}^{d-1}\sum_{\substack{i+n=0\\i,n\geq0}}^1E_{minj}(\phi)\partial_z^i\partial_t^nw_j-D_m\partial_zw_m-\gamma_m\partial_z\partial_tw_m,\psi_{3m}\right)_\Sigma\cr
\qquad=\left(\sum_{j=1}^{d-1}\sum_{\substack{i+n=0\\i,n\geq0}}^1E_{minj}(\phi)\partial_z^i\partial_t^nw_j-D_m\partial_zw_m-\gamma_m\partial_z\partial_tw_m,\partial_z\psi_{3m}\right)_Q,\nonumber
\end{align}
\end{subequations}
where $(f,g)_Q := \int_{t_0}^T\int_0^1f(z,t)g(z,t)\mathrm{d}z\mathrm{d}t$, $(f,g)_\Sigma := \int_{t_0}^Tf(1,t)g(1,t)-f(0,t)g(0,t)\mathrm{d}t$, $\psi_2\in L^2(Q)$ and $\psi_{1l},\psi_{3m}\in H^{1,0}(Q)$, and
\begin{subequations}
\begin{align}
\int_{t_0}^T\left(\left.\partial_z w_m\right|_{z=1}-A_1\left(\left.w_m\right|_{z=1}-\mathcal{W}\right)\right)\Psi_{1m}\mathrm{d}t&=0,\label{eq: 1wmw}\\
\int_{t_0}^T\left(\left.\partial_t w_m\right|_{z=0}-\hat{J}_m\frac{\mathcal{L}(\phi_{m,res}-\left.\phi_m\right|_{z=0})}{H_{1m}\left(\left.\phi\right|_{z=0}\right)}\right)\Psi_{2m}\mathrm{d}t&=0,\label{eq: 0wmw}\\
\text{and}\cr
\int_{t_0}^T\left(\partial_t \mathcal{W}-\left.v_3\right|_{z=1}-\hat{J}_d\frac{\mathcal{L}(\phi_{d,res}-\left.\phi_d\right|_{z=1})}{\Gamma\left(\left.\phi\right|_{z=1}\right)}\right)\Psi_3\mathrm{d}t&=0,\label{eq: heightw}
\end{align}
\end{subequations}
where $\Psi_{1m},\Psi_{2m},\Psi_3\in L^2(t_0,T)$ for all $1\leq m<d$.
\end{theorem}
\begin{proof}
\Cref{l: 1regul,l: LAS,l: 1limits}, \Cref{t: comp,t: timedomain}, and the Eberlein-Smulian theorem show that \Cref{a: ass3} is valid for a time domain $(t_0,T)$ and an upper bound $V>0$. Moreover, they show for this time domain that there exists a subsequence of $(\Delta t)$ converging to 0 such that both the linear and nonlinear terms of \Cref{eq: disc1,eq: disc2,eq: disc3} converge weakly to \Cref{eq: 1weak,eq: 2weak,eq: 3weak}, since both $\hat{\phi}_{l,\Delta t}$ and $\overline{\phi}_{l,\Delta t}$ converge strongly in $L^2(t_0,T;C^1(0,1))$ to $\phi_l$ for $l\in\{1,\ldots,d\}$, and all other variables have the necessary weak convergence.\\
The boundary conditions \cref{eq: 1wmw,eq: 0wmw,eq: heightw} follow immediately from \Cref{t: trace} together with \Cref{l: LAS,t: comp,l: bound} applied to \Cref{eq: rankhugBCD,eq: phiBCD}, since $\mathcal{L}(\cdot)\in W^{1,\infty}(0,1)$ and \Cref{a: ass4} imply $\mathcal{L}(\cdot)/H_{1m}(\cdot),\mathcal{L}(\cdot)/\Gamma(\cdot)\in W^{1,\infty}(0,1)$.
\end{proof}
\section{Conclusion}
We have proven the existence of physical weak solutions of the continuous system given by \Cref{eq: sys1,eq: sys2,eq: sys3} on the domain $[t_0,T]\times(0,1)$ with boundary conditions \cref{eq: phiBC,eq: rankhugBC} and initial conditions \cref{eq: init} and satisfying \Cref{a: ass1,a: ass4,a: ass5,a: ass6,a: ass7,a: ass9} by applying the Rothe method to the time discrete system given by \Cref{eq: disc1,eq: disc2,eq: disc3} with boundary conditions \cref{eq: phiBC,eq: rankhugBC} and initial conditions \cref{eq: init}, when the time interval size $T-t_0$ and a velocity $V$ are chosen such that $(T-t_0,V)\in S\cap\mathcal{R}\neq\emptyset$ with $S$ and $\mathcal{R}$ respectively as introduced in \Cref{l: tint,l: 1Vbound}.
\section*{Acknowledgments}$\;$\\
We acknowledge NWO for the MPE grant 657.000.004. Furthermore, we thank T. Aiki (Tokyo) and J. Zeman (Prague) for their discussions and contributions.
\appendix
\section{Derivation of quadratic inequalities}$\;$\\
\label{s: app0}
The discrete Gronwall inequalities, as stated in \Cref{l: 1gron,l: 1gron2}, apply to quadratic inequalities only. By testing the discrete system with suitable test functions one can obtain these quadratic inequalities with a cubic term. These cubic terms can often be transformed into quadratic terms by partial integration and application of \Cref{a: ass3}, which allows \Cref{l: 1gron} to be applied. However, in a single case we cannot transform the cubic term into quadratic terms. In this case the cubic term can be modified to fit the framework of \Cref{l: 1gron2}.\\
$\;$\\
We test \Cref{eq: disc1} successively with $\phi_l^k$ and $\mathcal{D}_{\Delta t}^k(\phi_l)$, which gives us, with use of $\partial_z\phi_l^k(0)=\partial_z\phi_l^k(1)=0$,
\begin{multline}
\label{eq: 1aprio1}
\frac{1}{2}\left[\mathcal{D}_{\Delta t}^k\left(\|\phi_l\|_{L^2}^2\right)+\Delta t\left\|\mathcal{D}_{\Delta t}^k\left(\phi_l\right)\right\|_{L^2}^2\right]+\delta_l\left\|\partial_z\phi_l^k\right\|_{L^2}^2\cr
\leq 2I_l\Gamma\sum_{n\neq d-1}\|\partial_z\phi_n^{k-1}\|_{L^2}\|v^{k-1}\|_{L^2}+I_l\Gamma\left\|\partial_z v^{k-1}\right\|_{L^2}\left\|\phi_l^k\right\|_{L^2}+G_{\phi,l}\left\|\phi_l^k\right\|_{L^2}\cr
+\sum_{m=1}^{d-1}\left[\sum_{i=0}^1\left(B_{li0m}\|\partial_z^iw_m^{k-1}\|_{L^2}+B_{li1m}\left\|\mathcal{D}_{\Delta t}^k\left(\partial_z^iw_m\right)\right\|_{L^2}\right)\left\|\phi_l^k\right\|_{L^2}\right.\cr
\left.+2\sum_{n\neq d-1}\|\partial_z\phi_n^{k-1}\|_{L^2}\left(B_{l10m}\|w_m^{k-1}\|_{L^2}+B_{l11m}\left\|\mathcal{D}_{\Delta t}^k\left(w_m\right)\right\|_{L^2}\right)\right]
\end{multline}
and
\begin{multline}
\label{eq: 1aprio2}
\left\|\mathcal{D}_{\Delta t}^k\left(\phi_l\right)\right\|_{L^2}^2+\frac{\delta_l}{2}\left[\mathcal{D}_{\Delta t}^k\left(\left\|\partial_z\phi_l\right\|_{L^2}^2\right)+\Delta t\left\|\mathcal{D}_{\Delta t}^k\left(\partial_z\phi_l\right)\right\|_{L^2}^2\right]\cr
\leq I_l\Gamma\left[2\!\!\sum_{n\neq d-1}\!\!\|\partial_z\phi_n^{k-1}\|_{L^2}\|v^{k-1}\|_{L^\infty}+\|\partial_zv^{k-1}\|_{L^2}\right]\left\|\mathcal{D}_{\Delta t}^k\left(\phi_l\right)\right\|_{L^2}
+G_{\phi,l}\left\|\mathcal{D}_{\Delta t}^k\left(\phi_l\right)\right\|_{L^2}\cr
+\sum_{m=1}^{d-1}\left[\sum_{i=0}^1\left(B_{li0m}\|\partial_z^iw_m^{k-1}\|_{L^2}+B_{li1m}\left\|\mathcal{D}_{\Delta t}^k\left(\partial_z^iw_m\right)\right\|_{L^2}\right)\left\|\mathcal{D}_{\Delta t}^k\left(\phi_l\right)\right\|_{L^2}\right.\cr
\left.+2\sum_{n\neq d-1}\|\partial_z\phi_n^{k-1}\|_{L^2}\left(B_{l10m}\|w_m^{k-1}\|_{L^\infty}+B_{l11m}\left\|\mathcal{D}_{\Delta t}^k\left(w_m\right)\right\|_{L^\infty}\right)\left\|\mathcal{D}_{\Delta t}^k\left(\phi_l\right)\right\|_{L^2}\right].
\end{multline}
Furthermore, we test \Cref{eq: disc3} successively with $w_m^k$ and $\mathcal{D}_{\Delta t}^k(w_m)$, in the evaluation of which we use the following results, derived from
\begin{enumerate}
\item the first boundary condition of \cref{eq: rankhugBCD}, which yields the estimate
\begin{equation}
\left|\mathcal{D}_{\Delta t}^k(w_m(0))\right|\leq\frac{\hat{J}_m\phi_{m,res}}{H_{\phi_{\min}}}\qquad\Rightarrow\qquad\left|w_m^k(0)\right|\leq\frac{\hat{J}_m\phi_{m,res}}{H_{\phi_{\min}}}(t_k-t_0)\nonumber
\end{equation}
\item the second boundary condition of \cref{eq: rankhugBCD}, leading to
\begin{multline}
\left|\partial_zw_m^k(1)\right|\leq|A_m|\left|w_m^k(1)-\mathcal{W}^k\right|\cr
\leq|A_m|\left(\frac{\hat{J}_m\phi_{m,res}}{H_{\phi_{\min}}}(t_k-t_0)+\|\partial_zw_m^k\|_{L^2}+|\mathcal{W}^k|\right)\nonumber
\end{multline}
\item the third boundary condition of \cref{eq: rankhugBCD}, which gives us
\begin{equation}
|v^k(z)|\leq\|\partial_zv^k\|_{L^2}
\end{equation}
\item the fourth boundary condition of \cref{eq: rankhugBCD}, resulting in
\begin{subequations}
\begin{align}
\left|\mathcal{D}_{\Delta t}^k(\mathcal{W})\right|&\leq \left\|\partial_z v^{k-1}\right\|_{L^2}+\frac{\hat{J}_d\phi_{d,res}}{\Gamma_{\phi_{\min}}}\label{eq: 1aprio6a}\qquad\text{and}\\
|\mathcal{W}^k|&\leq |\mathcal{W}^0|+\sum_{n=0}^{k-1}\left\|\partial_z v^{n}\right\|_{L^2}\Delta t+\frac{\hat{J}_d\phi_{d,res}}{\Gamma_{\phi_{\min}}}(t_k-t_0)\label{eq: 1aprio6b}\\
&\leq |\mathcal{W}^0|+V\sqrt{t_k-t_0}+\frac{\hat{J}_d\phi_{d,res}}{\Gamma_{\phi_{\min}}}(t_k-t_0).\nonumber
\end{align}
\end{subequations}
\item a first integral of \cref{eq: disc3}, written in short hand notation as
\begin{equation}
\mathcal{D}_{\Delta t}^k(w_m)-\partial_z\mathbb{S}_m^k=G_{w,m}(\phi^{k-1})-F_m(\phi^{k-1})v^{k-1},\nonumber
\end{equation}
where
\begin{subequations}
\begin{align}
\mathbb{S}_m^k&=\mathbb{S}_{m0}^k+\mathbb{S}_{m1}^k\cr
\mathbb{S}_{m0}^k&=\sum_{j=1}^{d-1}\left[E_{m00j}(\phi^{k-1})w_j^{k-1}+E_{m01j}\mathcal{D}_{\Delta t}^k(w_j)\right]\cr
\mathbb{S}_{m1}^k&=D_m\partial_zw_m^k+\gamma_m\mathcal{D}_{\Delta t}^k(\partial_zw_m)-\sum_{j=1}^{d-1}E_{m10j}(\phi^{k-1})\partial_zw_j^{k-1},\nonumber
\end{align}
\end{subequations}
stating that
\begin{equation}
\mathbb{S}_m^k(1)=\mathbb{S}_m^k(0)+\int_0^1\left(\mathcal{D}_{\Delta t}^k(w_m)-G_{w,m}(\phi^{k-1})+F_m(\phi^{k-1})v^{k-1}\right)\mathrm{d}z.\nonumber
\end{equation}
With this and the preceding results, we notice that $\mathbb{S}_{m0}^k(0)$, $\mathbb{S}_{m1}^k(1)$ and $\mathbb{S}_{m}^k(1)-\mathbb{S}_{m}^k(0)$ are "known" (i.e. can be expressed in known constants and the $L^2$-norms of the variables and their $\partial_z$-derivatives). All this gives us for the stockterm $\left.(\mathbb{S}_{m}^kw_m^k)\right|_0^1$, occurring in the equation resulting from the test of \cref{eq: disc3} with $w_m^k$,
\begin{multline}
\left.(\mathbb{S}_{m}^kw_m^k)\right|_0^1\!=\! \left[\mathbb{S}_{m}^k(1)-\mathbb{S}_{m}^k(0)\right]w_m^k(0)+\left[\mathbb{S}_{m0}^k(0)+\mathbb{S}_{m1}^k(1)\right]\!\!\left[w_m^k(1)-w_m^k(0)\right]\cr
+\left[\mathbb{S}_{m0}^k(1)-\mathbb{S}_{m0}^k(0)\right]\left[w_m^k(1)-w_m^k(0)\right],\nonumber
\end{multline}
where
\begin{subequations}
\begin{align}
\left|\mathbb{S}_{m0}^k(1)-\mathbb{S}_{m0}^k(0)\right|&\leq\sum_{j=1}^{d-1}\left[\vphantom{\frac{\phi}{\phi}}\left|E_{m00j}(\phi^{k-1}(1))\left(w_j^k(1)-w_j^k(0)\right)\right|\right.\cr
&\vphantom{a}\qquad\qquad+\left|\left(E_{m00j}(\phi^{k-1}(1))-E_{m00j}(\phi^{k-1}(0))\right)w_j^k(0)\right|\cr
&\vphantom{a}\qquad\qquad+\left|E_{m01j}(\phi^{k-1}(1))\left(\mathcal{D}_{\Delta t}^k(w_j(1))-\mathcal{D}_{\Delta t}^k(w_j(0))\right)\right|\cr
&\vphantom{a}\qquad\qquad\left.+\left|\left(E_{m01j}(\phi^{k-1}(1))-E_{m01j}(\phi^{k-1}(0))\right)\mathcal{D}_{\Delta t}^k(w_j(0))\right|\vphantom{\frac{\phi}{\phi}}\right]\cr
&\leq\sum_{j=1}^{d-1}\left[E_{m00j}\left(\|\partial_zw_j^k\|_{L^2}+\frac{\hat{J}_m\phi_{m,res}}{H_{\phi_{\min}}}(t_k-t_0)\right)\right.\cr
&\vphantom{a}\qquad\qquad\left.+E_{m01j}\left(\|\mathcal{D}_{\Delta t}^k(\partial_zw_j)\|_{L^2}+\frac{\hat{J}_m\phi_{m,res}}{H_{\phi_{\min}}}\right)\right],\nonumber
\end{align}
\end{subequations}
\end{enumerate}
and with use of the fundamental theorem of calculus to rewrite the boundary terms. Notice that this is justified by \Cref{t: 1dim}, which guarantees the existence of an absolutely continuous representative satisfying the fundamental theorem of calculus.
\begin{theorem}
\label{t: 1dim}
Let $I$ be an open, but possibly unbounded, interval. Let $u\in L^p(I)$ with $p\in(1,\infty]$, then $u\in W^{1,p}(I)$ iff $u$ is of bounded variation, i.e., there exists a constant $C$ such that for all $\phi\in C_c^1(I)$\footnote{We follow the notation of Brezis in \cite{Brezis2010} by introducing $C_c^k(\Omega)$ as the space of compactly supported, $k$ times continuously differentiable functions on $\Omega$.} we have the inequality    $\left|\int_Iu\phi' \right|\leq C\|\phi\|_{L^{p'}(I)}$.
Furthermore, we can take $C = \|u'\|_{L^p(I)}$, \cite{Brezis2010}.\\
Moreover, $u\in W^{1,p}(I)$ with $p\in[1,\infty]$ iff there exists an absolutely continuous representative of $u$ in $L^p(I)$ with a classical derivative in $L^p(I)$, \cite{Brezis2010,Leoni2009}.
\end{theorem}
All this results in
\begin{multline}
\label{eq: 1aprio3}
\frac{1}{2}\left[\mathcal{D}_{\Delta t}^k\left(\|w_m\|_{L^2}^2\right)+\Delta t\left\|\mathcal{D}_{\Delta t}^k\left(w_m\right)\right\|_{L^2}^2\right]+D_m\left\|\partial_zw_m^k\right\|_{L^2}^2\cr
+\frac{\gamma_m}{2}\left[\mathcal{D}_{\Delta t}^k\left(\|\partial_zw_m\|_{L^2}^2\right)+\Delta t\left\|\mathcal{D}_{\Delta t}^k\left(\partial_zw_m\right)\right\|_{L^2}^2\right]\cr
\leq L^k_m\|w_m^k\|_{L^2}+M^k_m\|\partial_zw_m^k\|_{L^2}+N^k_m(t_k-t_0)
\end{multline}
and
\begin{multline}
\label{eq: 1aprio4}
\left\|\mathcal{D}_{\Delta t}^k\left(w_m\right)\right\|_{L^2}^2+\gamma_m\left\|\mathcal{D}_{\Delta t}^k\left(\partial_zw_m\right)\right\|_{L^2}^2\cr
+\frac{D_m}{2}\left[\mathcal{D}_{\Delta t}^k\left(\|\partial_zw_m\|_{L^2}^2\right)+\Delta t\left\|\mathcal{D}_{\Delta t}^k\left(\partial_zw_m\right)\right\|_{L^2}^2\right]\cr
\leq L^k_m\left\|\mathcal{D}_{\Delta t}^k\left(w_m\right)\right\|_{L^2}+M^k_m\left\|\mathcal{D}_{\Delta t}^k\left(\partial_zw_m\right)\right\|_{L^2}+N^k_m
\end{multline}
For brevity, we have omitted to point out the spatial domain dependence in the norms. The newly introduced functions $L_m^k$, $M_m^k$ and $N_m^k$, are defined as
\begin{subequations}
\begin{align}
L^k_m&=F_m\|v^{k-1}\|_{L^2}+G_{w,m}\label{eq: Lbound}
\end{align}
\end{subequations}
\begin{subequations}
\begin{align}
M^k_m&=\sum_{j=1}^{d-1}\left[E_{m01j}\left\|\mathcal{D}_{\Delta t}^k\left(w_j\right)\right\|_{L^2}+\sum_{i=0}^1E_{mi0j}\|\partial_z^iw_j^{k-1}\|_{L^2}\right]\label{eq: 1aprio5b}\\
&+\sum_{j=1}^{d-1}\left[E_{m10j}|A_j|\left(\frac{\hat{J}_j\phi_{j,res}}{H_{\phi_{\min}}}(t_k-t_0)+ \|\partial_zw_j^{k-1}\|_{L^2}\right.\right.\cr
&\vphantom{a}\qquad\qquad\qquad\qquad\left.\left.+V\sqrt{t_k-t_0}+\frac{\hat{J}_d\phi_{d,res}}{\Gamma_{\phi_{\min}}}(t_k-t_0)\right)\right]\cr
&+\sum_{j=1}^{d-1}\left[E_{m00j}\left(2\frac{\hat{J}_j\phi_{j,res}}{H_{\phi_{\min}}}(t_k-t_0)+ \|\partial_zw_j^{k-1}\|_{L^2}\right)\right]\cr
&+\sum_{j=1}^{d-1}\left[E_{m01j}\left(2\frac{\hat{J}_j\phi_{j,res}}{H_{\phi_{\min}}}+ \|\mathcal{D}_{\Delta t}^k(\partial_zw_j)\|_{L^2}\right)\right]\cr
&+D_m|A_m|\left(\frac{\hat{J}_m\phi_{m,res}}{H_{\phi_{\min}}}(t_k-t_0)+ \|\partial_zw_m^{k}\|_{L^2}\right.\cr
&\vphantom{a}\qquad\qquad\left.+V\sqrt{t_k-t_0}+\frac{\hat{J}_d\phi_{d,res}}{\Gamma_{\phi_{\min}}}(t_k-t_0)\right)\cr
&+\gamma_m|A_m|\left(\frac{\hat{J}_m\phi_{m,res}}{H_{\phi_{\min}}}+ \|\mathcal{D}_{\Delta t}^k(\partial_zw_m)\|_{L^2}+\|\partial_zv^{k-1}\|_{L^2}+\frac{\hat{J}_d\phi_{d,res}}{\Gamma_{\phi_{\min}}}\right)\cr
&=\mathcal{M}_m^k+\sum_{j=1}^{d-1}\sum_{i=0}^1\left[\mathcal{M}_{mi0j}\|\partial_z^iw_j^{k-1}\|_{L^2}+\mathcal{M}_{mi1j}\left\|\mathcal{D}_{\Delta t}^k(\partial_z^iw_j)\right\|_{L^2}\right]\label{eq: Mbound}\\
&+\mathcal{M}_{mm}\|\partial_zw_m^{k}\|_{L^2}+\mathcal{M}_{vm}\|\partial_zv^{k-1}\|_{L^2}\nonumber
\end{align}
\end{subequations}
\begin{subequations}
\begin{align}
N^k_m&=\left(G_{w,m}+F_m\|\partial_zv^{k-1}\|_{L^2}+\|\mathcal{D}_{\Delta t}^k(w_m)\|_{L^2}\right)\frac{\hat{J}_m\phi_{m,res}}{H_{\phi_{\min}}}\label{eq: 1aprio5c}\\
&= \mathcal{N}_{m0}+\mathcal{N}_{m1}\|\partial_zv^{k-1}\|_{L^2}+\mathcal{N}_{m2}\|\mathcal{D}_{\Delta t}^k(w_m)\|_{L^2}\label{eq: Nbound}
\end{align}
\end{subequations}
Remark that the $L^\infty$ norms are sufficient, since they can be bounded from above by \Cref{a: ass3} or the embedding $H^1(\Omega)\hookrightarrow L^\infty(\Omega)$ for bounded intervals $\Omega$ with embedding constant $C_\infty = \sqrt{|\Omega|+1/|\Omega|}$, which equals $\sqrt{2}$ for our domain $\Omega = (0,1)$.\\
We combine the different quadratic inequalities to create inequalities for which we can apply Gronwall's lemmas \cref{l: 1gron,l: 1gron2}. We add \Cref{eq: 1aprio3,eq: 1aprio4} and apply Young's inequality with a parameter $\eta_{index}>0$, Minkowski's inequality and the $H^1(\Omega)\hookrightarrow L^\infty(\Omega)$ embedding, leading to the new quadratic inequality:
\begin{multline}
\frac{1}{2}\mathcal{D}_{\Delta t}^k\left(\sum_{m=1}^{d-1}\left[\|w_m\|_{L^2}^2+(\gamma_m+D_m)\|\partial_zw_m\|_{L^2}^2\right]\right)\label{eq: wboundsA}\\
+\sum_{m=1}^{d-1}\left[\left(1+\frac{\Delta t}{2}\right)\left\|\mathcal{D}_{\Delta t}^k(w_m)\right\|_{L^2}^2+\left[\gamma_m\left(1+\frac{\Delta t}{2}\right)+D_m\frac{\Delta t}{2}\right]\left\|\mathcal{D}_{\Delta t}^k\left(\partial_zw_m\right)\right\|_{L^2}^2\right]\cr
\leq K_{w0}+\!\!\sum_{m=1}^{d-1}\!\left[\vphantom{\frac{A}{A}} K_{w1m}\|w_m^k\|_{L^2}^2\!+\!K_{w2m}\|\partial_zw_m^k\|_{L^2}^2\!+\!K_{w3m}\|w_m^{k-1}\|_{L^2}^2\!+\!K_{w4m}\|\partial_zw_m^{k-1}\|_{L^2}^2\right.\cr
\left.\vphantom{\frac{A}{A}}+\!K_{w5m}\left\|\mathcal{D}_{\Delta t}^k(w_m)\right\|_{L^2}^2\!+\!K_{w6m}\left\|\mathcal{D}_{\Delta t}^k(\partial_zw_m)\right\|_{L^2}^2\right]+K_{w7}\|v^{k-1}\|_{L^2}^2+K_{w8}\|\partial_zv^{k-1}\|_{L^2}^2,\!\!\!
\end{multline}
Analogously, from \Cref{eq: 1aprio1,eq: 1aprio2}, the latter summed over $l$, we obtain the quadratic inequalities:
\begin{multline}
\mathcal{D}_{\Delta t}^k\left(\|\phi_l\|_{L^2}^2\right)+2\delta_l\|\partial_z\phi_l^k\|_{L^2}^2+\Delta t\left\|\mathcal{D}_{\Delta t}^k(\phi_l)\right\|_{L^2}^2\label{eq: phiboundsA}\\
\leq K_{a\phi0}+K_{a\phi1}\|\partial_zv^{k-1}\|_{L^2}^2+K_{a\phi2l}\|\phi_l^{k}\|_{L^2}^2+\sum_{n\neq d-1}\left[K_{a\phi3ln}\|\partial_z\phi_n^{k-1}\|_{L^2}^2\right]\cr
+\sum_{m=1}^{d-1}\sum_{i=0}^1\left[K_{a\phi(4+i)m}\left\|\partial_z^iw_m^{k-1}\right\|_{L^2}^2+K_{a\phi(6+i)m}\left\|\mathcal{D}_{\Delta t}^k(\partial_z^iw_m)\right\|_{L^2}^2\right],
\end{multline}
and
\begin{multline}
\mathcal{D}_{\Delta t}^k\left(\sum_{l\neq d-1}\|\partial_z\phi_l\|_{L^2}^2\right)+\sum_{l\neq d-1}\left[\frac{2}{\delta_l}\left\|\mathcal{D}_{\Delta t}^k(\phi_l)\right\|_{L^2}^2+\Delta t\left\|\mathcal{D}_{\Delta t}^k\left(\partial_z\phi_l\right)\right\|_{L^2}^2\right]\label{eq: phibounds1A}\\
\leq K_{b\phi0}+K_{b\phi1}\|\partial_zv^{k-1}\|_{L^2}^2+\sum_{l\neq d-1}\!\left[ K_{b\phi2l}\left\|\mathcal{D}_{\Delta t}^k(\phi_l)\right\|_{L^2}^2+K_{b\phi3l}^{k-1}\|\partial_z\phi_l^{k-1}\|_{L^2}^2\right]\cr
+\sum_{m=1}^{d-1}\sum_{i=0}^1\left[K_{b\phi(4+i)m}\left\|\partial_z^iw_m^{k-1}\right\|_{L^2}^2+K_{b\phi(6+i)m}\left\|\mathcal{D}_{\Delta t}^k(\partial_z^iw_m)\right\|_{L^2}^2\right].
\end{multline}
The constants $K_{w\,index}$ are given by
\begin{subequations}
\begin{align}
K_{w0}^k&:=\sum_{m=1}^{d-1}\left[\frac{G_{w,m}^2}{2}\left(\frac{1}{\eta_{1m}}+\frac{1}{\eta_{2m}}\right)+\mathcal{N}_{m0}\left[1+(t_k-t_0)\right]\right.\label{eq: kw0}\\
&+\frac{\mathcal{N}_{m1}^2}{2}\left(\frac{1}{\eta_{\mathcal{N}m11}}+\frac{t_k-t_0}{\eta_{\mathcal{N}m12}}\right)+\frac{\mathcal{N}_{m2}^2}{2}\left(\frac{1}{\eta_{\mathcal{N}m21}}+\frac{t_k-t_0}{\eta_{\mathcal{N}m22}}\right)\cr
&\left.+\frac{(\mathcal{M}_{m}^k)^2}{2}\left(\frac{1}{\eta_{\mathcal{M}m1}}+\frac{1}{\eta_{\mathcal{M}m2}}\right)\right]\cr
K_{w1m}&:=\frac{\eta_{1m}}{2}+\frac{\eta_{Lm1}}{2}\label{eq: kw1},\\
K_{w2m}&:=\frac{\eta_{\mathcal{M}m1}}{2}+\sum_{j=1}^{d-1}\sum_{i=0}^{1}\left(\frac{\eta_{\mathcal{M}mi0j0}}{2}+\frac{\eta_{\mathcal{M}mi1j0}}{2}\right)+\frac{\eta_{\mathcal{M}mm}}{2}\label{eq: kw2}\\
&+\frac{\mathcal{M}_{vm}^2}{2\eta_{\mathcal{M}vm1}}+\mathcal{M}_{mm}-D_m,\cr
K_{w3m}&:=\sum_{j=1}^{d-1}\left[\frac{\mathcal{M}_{j00m}^2}{2\eta_{\mathcal{M}j00m0}}+\frac{\mathcal{M}_{j00m}}{2}\eta_{\mathcal{M}j00m1}\right]\label{eq: kw3},\\
K_{w4m}&:=\sum_{j=1}^{d-1}\left[\frac{\mathcal{M}_{j10m}^2}{2\eta_{\mathcal{M}j10m0}}+\frac{\mathcal{M}_{j10m}}{2}\eta_{\mathcal{M}j10m1}\right]\label{eq: kw4},\\
K_{w5m}&:=\frac{\eta_{2m}}{2}+\frac{\eta_{Lm2}}{2}+\frac{\eta_{\mathcal{N}m21}}{2}+\frac{\eta_{\mathcal{N}m22}}{2}\label{eq: kw5}\\
&+\sum_{j=1}^{d-1}\left[\frac{\mathcal{M}_{j01m}^2}{2\eta_{\mathcal{M}j01m0}}+\frac{\mathcal{M}_{j01m}}{2}\eta_{\mathcal{M}j01m1}\right],\cr
K_{w6m}&:=\frac{\eta_{\mathcal{M}m2}}{2}+\frac{\mathcal{M}_{mm}^2}{2\eta_{\mathcal{M}mm}}+\frac{\mathcal{M}_{vm}^2}{2\eta_{\mathcal{M}vm2}}+\sum_{j=1}^{d-1}\left[\frac{\mathcal{M}_{m01j}}{2\eta_{\mathcal{M}m01j1}}+\sum_{i=0}^1\frac{\mathcal{M}_{mi0j}}{2\eta_{\mathcal{M}mi0j1}}\right]\label{eq: kw6}\\
&+\sum_{j=1}^{d-1}\left[\frac{\mathcal{M}_{j11m}^2}{2\eta_{\mathcal{M}j11m0}}+\frac{\mathcal{M}_{m11j}}{2}\eta_{\mathcal{M}m11j1}+\frac{\mathcal{M}_{j11m}}{2\eta_{\mathcal{M}j11m1}}\right],\cr
K_{w7}&:=\sum_{m=1}^{d-1}\frac{F_m^2}{2}\left(\frac{1}{\eta_{Lm1}}+\frac{1}{\eta_{Lm2}}\right)\label{eq: kw7},\\
K_{w8}&:=\sum_{m=1}^{d-1}\left(\frac{\eta_{\mathcal{N}m11}}{2}+\frac{\eta_{\mathcal{N}m12}}{2}+\frac{\eta_{\mathcal{M}vm1}}{2}+\frac{\eta_{\mathcal{M}vm2}}{2}\right)\label{eq: kw8},
\end{align}
\end{subequations}
where $\eta_{index}>0$ are auxiliary constants found in Cauchy's inequality, while the constants $K_{x\phi\,index}$ for $x\in\{a,b\}$ are equal to
\begin{subequations}
\begin{align}
K_{x\phi0}&:=\left(\delta_{bx}\sum_{l\neq d-1}+\delta_{ax}\max_{l\neq d-1}\right)\frac{G_{\phi,l}^2}{\sigma_{x,l}^2\eta_{xgl}}\label{eq: xkphi0},\\
K_{x\phi1}&:=\left(\delta_{bx}\sum_{l\neq d-1}+\delta_{ax}\max_{l\neq d-1}\right)\left[\frac{I_l^2\Gamma^2}{\sigma_{x,l}^2\eta_{xl}}+2\delta_{ax}\sum_{n\neq d-1}\frac{I_l^2\Gamma^2C_p^2}{\eta_{aln}}\right]\label{eq: xkphi1},\\
K_{x\phi2l}&:=\eta_{xl}+\eta_{xgl}+\sum_{m=1}^{d-1}\left(\eta_{xl00m}+\eta_{xl01m}+\eta_{xl10m}+\eta_{xl11m}\right)\label{eq: xkphi2}\\
&+2\delta_{bx}\sum_{n\neq d-1}\left[\eta_{bln}+\sum_{j=1}^{d-1}\left(\eta_{bl10jn}+\eta_{bl11jn}\right)\right],\cr
K_{a\phi3ln}&:=2\left[\eta_{aln}+\sum_{m=1}^{d-1}\left(\eta_{al10mn}+\eta_{al11mn}\right)\right]\label{eq: xkphi3a},\\
K_{b\phi3l}^{k-1}&:=2\sum_{n\neq d-1}\left[\frac{I_n^2\Gamma^2C_{\infty}^2}{\delta_n^2\eta_{bnl}}\|v^{k-1}\|_{H^1}^2\right.\label{eq: xkphi3b}\\
&\qquad\qquad\left.+\sum_{m=1}^{d-1}\left(\frac{B_{n11m}^2C_{\infty}^2}{\delta_l^2\eta_{bn11ml}}\left\|\mathcal{D}_{\Delta t}^k(w_m)\right\|_{H^1}^2+\frac{B_{n10m}^2C_{\infty}^2}{\delta_l^2\eta_{bn10ml}}\|w_m^{k-1}\|_{H^1}^2\right)\right],\cr
&=K_{\phi3l1}\left\|v^{k-1}\right\|_{H^1}^2+\sum_{m=1}^{d-1}\left[K_{\phi3l2m}\left\|\mathcal{D}_{\Delta t}^k(w_m)\right\|_{H^1}^2+K_{\phi3l3m}\|w_m^{k-1}\|_{H^1}^2\right],\cr
K_{x\phi4m}&:=\left(\delta_{bx}\sum_{l\neq d-1}+\delta_{ax}\max_{l\neq d-1}\right)\left[\frac{B_{l00m}^2}{\sigma_{x,l}^2\eta_{xl00m}}+2\delta_{ax}\sum_{n\neq d-1}\frac{B_{l10m}^2}{\eta_{al10mn}}\right]\label{eq: xkphi4},\\
K_{x\phi5m}&:=\left(\delta_{bx}\sum_{l\neq d-1}+\delta_{ax}\max_{l\neq d-1}\right)\frac{B_{l10m}^2}{\sigma_{x,l}^2\eta_{xl10m}}\label{eq: xkphi5},\\
K_{x\phi6m}&:=\left(\delta_{bx}\sum_{l\neq d-1}+\delta_{ax}\max_{l\neq d-1}\right)\left[\frac{B_{l01m}^2}{\sigma_{x,l}^2\eta_{xl01m}}+2\delta_{ax}\sum_{n\neq d-1}\frac{B_{l11m}^2}{\eta_{al11mn}}\right]\label{eq: xkphi6},\\
K_{x\phi7m}&:=\left(\delta_{bx}\sum_{l\neq d-1}+\delta_{ax}\max_{l\neq d-1}\right)\frac{B_{l11m}^2}{\sigma_{x,l}^2\eta_{xl11m}}\label{eq: xkphi7},
\end{align}
\end{subequations}
where $\sigma_{x,l}$ equals $1$ for $x=a$ and $\delta_l$ for $x=b$ and where $\delta_{ij}$ denotes the Kronecker symbol. Furthermore, we have used the Poincar\'{e} inequality for $v^{k-1}$ with the constant $C_p = |\Omega|=1$.\\
$\;$\\
With \Cref{eq: phibounds1A,eq: phiboundsA,eq: wboundsA} we can bound the $H^1(0,1)$ norms of functions $w_1^k,\ldots,w_{d-1}^k$, $\phi_1^k,\ldots,\phi_{d-2}^k$ and $\phi_d^k$. The $H^1(0,1)$ norm bounds of $\phi_{d-1}^k$ and $v^k$ need to be determined in a different way.\\
$\;$\\
Even though one can use the identity $\sum_{l=1}^{d}\phi_l^k=1$ to obtain a bound for $\|\phi_{d-1}^k\|_{L^2}$, this identity will give only the universal upper bound $1-(d-1)\phi_{\min}$ as we have no knowledge for the lower bounds of $\|\phi_l^k\|_{L^2}$ except for that this value is greater or equal to $\phi_{\min}$. Therefore, we create a differential equation for $\phi_{d-1}$ by applying $\sum_{l=1}^{d}\phi_l^k=1$ to \Cref{eq: disc1}. Then, we test with $\phi_{d-1}^k$ to obtain a non-trivial inequality. Finally, we can apply \Cref{l: 1gron} to obtain an upper bound to the $L^2(0,1)$ norm of $\phi_{d-1}^k$. All this results in
\begin{multline}
\mathcal{D}_{\Delta t}^k\left(\|\phi_{d-1}\|_{L^2}^2\right)+\Delta t\left\|\mathcal{D}_{\Delta t}^k\left(\phi_{d-1}\right)\right\|_{L^2}^2\label{eq: phid}\cr
\leq 2\left\|\phi_{d-1}^k\right\|_{L^2}\sum_{l\neq d-1}\left[2I_l\Gamma\sum_{n\neq d-1}\|\partial_z\phi_n^{k-1}\|_{L^2}\|v^{k-1}\|_{L^\infty}+I_l\Gamma\left\|\partial_z v^{k-1}\right\|_{L^2}\right.\cr
\left.+G_{\phi,l}+\sum_{m=1}^{d-1}\left[\sum_{i=0}^1B_{li0m}\|\partial_z^iw_m^{k-1}\|_{L^2}+B_{li1m}\left\|\mathcal{D}_{\Delta t}^k\left(\partial_z^iw_m\right)\right\|_{L^2}\right]\right]\cr
+4\sum_{l\neq d-1}\sum_{m=1}^{d-1}\left[\sum_{n\neq d-1}\|\partial_z\phi_n^{k-1}\|_{L^2}\left(B_{l10m}\|w_m^{k-1}\|_{L^2}+B_{l11m}\left\|\mathcal{D}_{\Delta t}^k\left(w_m\right)\right\|_{L^2}\right)\right]\cr
+2\left\|\partial_z\phi_{d-1}^k\right\|_{L^2}\sum_{l\neq d-1}\delta_l\left\|\partial_z\phi_l^k\right\|_{L^2}
\end{multline}
Fortunately, we can use $\sum_j\phi_j^k=1$ to obtain directly an upper bound for $\|\partial_z\phi_{d-1}^k\|_{L^2}$.
\begin{equation}
\left\|\partial_z\phi_{d-1}^k\right\|_{L^2}^2\leq(d-1)\sum_{l\neq d-1}\left\|\partial_z\phi_l^k\right\|_{L^2}^2\label{eq: phidz}
\end{equation}
The bounds for $\left\|\partial_z \phi_l^{k}\right\|_{H^1}$ follow from \Cref{t: 1dim} applied to \Cref{eq: disc1} together with the previous upper bounds. The upper bounds for $\left\|\partial_z \phi_l^{k}\right\|_{H^1}$ equal
\begin{subequations}\begin{align}
\left\|\partial_z \phi_l^{k}\right\|_{H^1}^2 &= \left\|\partial_z \phi_l^{k}\right\|_{L^2}^2+\left\|\partial^2_z \phi_l^{k}\right\|_{L^2}^2\cr
\text{with}\qquad\!\!\quad\cr
\left\|\partial^2_z \phi_l^{k}\right\|_{L^2}&\leq\frac{1}{\delta_l}\left[\|\mathcal{D}_{\Delta t}^{k}(\phi_l)\|_{L^2}+I_l\Gamma\|\partial_zv^{k-1}\|_{L^2}\!\vphantom{\sum_{i=0}^1}\right.\cr
&+G_{\phi,l}+2I_l\Gamma C_\infty\|v^{k-1}\|_{H^1}\left(\sum_{n\neq d-1}\|\partial_z\phi_n^{k-1}\|_{L^2}\right)\cr
&+\sum_{m=1}^{d-1}\left[\sum_{i=0}^1\left(B_{li0m}\|\partial_z^iw_m^{k-1}\|_{L^2}+B_{li1m}\|\mathcal{D}_{\Delta t}^k(\partial_z^iw_m)\|_{L^2}\right)\right.\cr
&\left.+\left.2C_\infty\!\!\left(\sum_{n\neq d-1}\|\partial_z\phi_n^{k-1}\|_{L^2}\right)\left(B_{l10m}\|w_m^{k-1}\|_{H^1}+B_{l11m}\|\mathcal{D}_{\Delta t}^k(w_m)\|_{H^1}\right)\right]\right]\nonumber
\end{align}
\end{subequations}
Therefore, the sum $\sum_{j=1}^k\left\|\partial^2_z \phi_l^{j}\right\|_{L^2}^2\Delta t$ can be bounded independent of $\Delta t$.\\
A similar approach to determine bounds for $\|\partial_zw_m^k\|_{H^1}$ leads to a $\Delta t$-independent upper bound for $\|\partial_zw_m^k\|_{H^1}$ itself.\\
$\;$\\
For the $H^1(0,1)$ norm of $v^k$ we will use the previously determined inequalities and apply them to a rewritten version of \Cref{eq: disc2} leading to the following inequality.
\begin{subequations}\begin{align}
\sum_{j=0}^k&\left\|\partial_z v^j\right\|_{L^2}^2\Delta t\cr
&\leq\left\|\partial_z v^0\right\|_{L^2}^2\Delta t+\frac{\Delta t}{\Gamma(\phi^0)^2}\left\|\partial_z\left(\Gamma(\phi^0)v^1\right)\right\|_{L^2}^2\cr
&+\sum_{j=2}^k\frac{\Delta t}{\Gamma_{\phi_{\min}}^2}\!\!\left(\left\|\partial_z\!\!\left(\Gamma(\phi^{j-1})v^j\right)\right\|_{L^2}+C_{\infty}\Gamma\!\!\sum_{l\neq d-1}\!\!\left\|\partial_z \phi_l^{j-1}\right\|_{L^2}\left\|v^{j}\right\|_{H^1}\right)^2\cr
&\leq\left\|\partial_z v^0\right\|_{L^2}^2\Delta t+\frac{d\Delta t}{\Gamma(\phi^0)^2}\left[G_v(\phi^0)^2+\sum_{m=1}^{d-1}H_{1m}(\phi^0)^2\|\mathcal{D}_{\Delta t}^1(\partial_zw_m)\|_{L^2}^2\right]\cr
&+\sum_{j=2}^k\frac{(4d-2)\Delta t}{\Gamma_{\phi_{\min}}^2}\!\!\left(8C_{\infty}^2\Gamma^2(d-1)\left[\sum_{l\neq d-1}\!\!\left\|\partial_z \phi_l^{j-1}\right\|_{L^2}^2\right]\left\|\partial_zv^{j}\right\|_{L^2}^2 +G_v^2\right.\cr
&+\sum_{m=1}^{d-1}\left[H_{0m}^2\|\partial_zw_m^{j-1}\|_{L^2}^2+H_{1m}^2\|\mathcal{D}_{\Delta t}^j(\partial_zw_m)\|_{L^2}^2\vphantom{\sum_{n=1}^{d-1}}\right.\cr
&\left.+\left.4C_\infty^2(d-1)\left(\sum_{n\neq d-1}\|\partial_z\phi_n^{j-1}\|_{L^2}^2\right)\left(H_{0m}^2\|w_m^{j-1}\|_{H^1}^2+H_{1m}^2\|\mathcal{D}_{\Delta t}^j(w_m)\|_{H^1}^2\right)\right]\right)\cr
&\leq Q_{\Delta t}(T_{\Delta t}-t_0,V^2)\nonumber
\end{align}
\end{subequations}
Since all previous upper bounds are strictly increasing and continuous in $T_{\Delta t}-t_0$ and $V^2$, these properties hold as well for $Q_{\Delta t}(T_{\Delta t}-t_0,V^2)$.
\bibliographystyle{siamplain}
\bibliography{references1}
\end{document}